\documentclass[11pt,reqno]{amsart}
\usepackage{a4wide,fullpage}
\usepackage{amssymb,amsmath,amsthm,mathrsfs,bm,bbm}
\usepackage{comment,enumitem,url}
\usepackage{xcolor}
\usepackage{cleveref}
\usepackage{mathtools}
\usepackage{bbm}
\newtheorem{theorem}{Theorem}
\newtheorem{lemma}[theorem]{Lemma}
\newtheorem{corollary}[theorem]{Corollary}

\newtheorem{proposition}[theorem]{Proposition}

\newtheorem*{remark}{Remark}

\theoremstyle{remark}

\theoremstyle{definition}

\numberwithin{theorem}{section} \numberwithin{equation}{section}

\setlist[enumerate]{leftmargin=*,label=\rm{(\arabic*)}}
\setlist[itemize]{leftmargin=*}

\newcommand{\ord}{\text {\rm ord}}

\newcommand{\Li}{\mathrm{Li}}

\newcommand{\R}{\mathbb{R}}
\newcommand{\C}{\mathbb{C}}

\newcommand{\Z}{\mathbb{Z}}
\newcommand{\N}{\mathbb{N}}

\newcommand{\SL}{{\text {\rm SL}}}

\newcommand{\PSL}{{\text {\rm PSL}}}

\newcommand{\re}{\mathrm{Re}}
\newcommand{\im}{\mathrm{Im}}

\newcommand{\Log}{\mathrm{Log}}

\newcommand{\hyp}{\mathrm{hyp}}
\newcommand{\vol}{\mathrm{vol}}
\def\H{\mathbb{H}}

\renewcommand{\pmod}[1]{\ \, \left( \mathrm{mod} \, #1 \right)}

\newcommand{\zf}{\mathfrak{z}}

\allowdisplaybreaks

\author{Kathrin Bringmann}
\address{Department of Mathematics and Computer Science, Division of Mathematics, University of Cologne, Weyertal 86--90, 50931 Cologne, Germany}
\email{kbringma@math.uni-koeln.de}
\author{Jay Jorgenson}
\address{Department of Mathematics, The City College of New York, Convent Avenue at 138th Street, New York, NY 10031, USA}
\email{jjorgenson@mindspring.com}
\author{Lejla Smajlovi\'{c}}
\address{Department of Mathematics, University of Sarajevo, Zmaja od Bosne 35, 71 000 Sarajevo, Bosnia and Herzegovina}
\email{lejlas@pmf.unsa.ba}

\begin{document}

\title{On a generating function of Niebur--Poincar\'e series}
%\footnote{The second named author acknowledges grant support from several PSC-CUNY Awards, which are jointly funded
%by the Professional Staff Congress and The City University of New York.}
\date{\today}

\begin{abstract}
Let $\Gamma\subset \PSL_2(\R)$ be a Fuchsian group of the first kind which has a cusp $i\infty$ of width one.
In this paper, we first consider a generating function formed with the Niebur--Poincar\'e series
$\{F_{m,s}(\tau)\}_{m\ge 1}$ associated to $i\infty$.  %The generating series is defined to be a function of one complex variable $s$ with $\re(s)>1$ and
%two variables $z$ and $\tau$ in the upper half plane $\mathbb{H}$.  The series is absolutely and uniformly convergent for $\re(s)>1$ and,
%is often the case in analytic number theory, there is specific interest to its analytic continuation to $s=1$.
We prove a relation between the continuation of this generating function to $s=1$
with the resolvent kernel associated to the hyperbolic
Laplacian and the non-holomorphic Eisenstein series associated to $i\infty$, also at $s=1$.  %Going further, it is proved that
%as a function of $z$ (resp. $\tau$), the generating function at $s=1$ is a weight two (resp. weight zero) polar harmonic Maass form with
%explicitly computable poles.
Secondly, we show that, for any $s\in \mathbb{N}$, the generating function equals Poincar\'e type series involving polylogarithms.
%If $s=1$, the series over $\Gamma_{\infty}\backslash \Gamma$
%is conditionally convergent while for other integer values of $s$ the series is absolutely convergent.  Third,
%we then consider a different generating function formed with derivatives in $s$ of the Niebur-Poincar\'e series.  In this case, we prove that the continuation
%of the generating series at $s=1$ is the
%$\Gamma$-periodization of a point-pair invariant involving the Rogers dialogarithm and the Kronecker limit function associated
%to the non-holomorphic Eisenstein series.
We also consider a generating function formed with derivatives in $s$ of the Niebur--Poincar\'e series and prove that the continuation
of the generating function at $s=1$ can be expressed in terms of
$\Gamma$-periodization of a point-pair invariant involving the Rogers dilogarithm and the Kronecker limit function associated to the non-holomorphic Eisenstein series.
%Throughout this article our analysis employs the spectral theory of
%automorphic forms and is applicable to any Fuchsian group; in particular, $\Gamma$ need not be arithmetic.
%Furthermore, our computations
%of the regularized inner product use spectral expansions rather than Fourier expansions in the cusp.  As far as we know, the point-pair
%invariant involving the Rogers dilogarithm is novel within the field of automorphic forms.
%\textcolor{blue}{A few selling points, not to forget when writing intro: (a) our approach allows us to deduce results for any Fuchsian group $\Gamma$ as the %main feature of the study is that all objects studied are actually eigenfunctions of the Laplacian; (b) the computation of the inner product uses spectral %expansion instead of Fourier expansions; the computation of the inner product using Fourier expansions might yield different expressions and hence new %identities - to be studied in another paper (c) we deduce a novel automorphic form which is given in terms of the point pair invariant involving the Rogers %dilogarithm}
\end{abstract}

%{\kbf TD: Make variables consistent: use $;$ for distinguishing space and complex variable (e.g. $E(z;s)$) use $,$ for space variables (e.g. $G_s(z,\tau)$.} %{\bf KB: I think we should more tell the reader what we use. I started to do that.}

\maketitle

\section{Introduction and statement of results}

\subsection{Definitions} Let $\Gamma\subset \PSL_2(\R)$ be a Fuchsian group of the first kind which has a cusp $i\infty$ of
width one\footnote{Actually the surface $M:=\Gamma \backslash \mathbb{H}$ has the cusp $i\infty$ of width one.} and possibly other cusps
as well as torsion points.
For $m\in\Z\setminus\{0\}$, the {\it Niebur--Poincar\'e series} (see \cite{Ni73}) associated to $i\infty$ is defined,
for $s\in\C$ with $\re(s)>1$, by
\begin{equation}\label{eq. N-P series defn}
F_{m,s}(\tau):=\sum_{\gamma\in\Gamma_\infty\setminus\Gamma} \sqrt{\im(\gamma\tau)}I_{s-\frac12}\left(2\pi |m|\im(\gamma \tau)\right)e^{2\pi im\re(\gamma\tau)},
\end{equation}
where $\Gamma_\infty:=\{\begin{psmallmatrix}1&n\\0&1\end{psmallmatrix}:n\in\Z\}$
is the stabilizer of $i\infty$ in $\Gamma$
and $I_\kappa$ denotes the $I$-Bessel function of order $\kappa$. The functions $F_{m,s}$ have meromorphic continuations
to the entire $s$-plane which are holomorphic at $s=1$ (see \cite{Ne73, Ni73}).
%{\bf KB: Should give a reference here. Maybe Niebur?}
In this paper we study the following generating function
\begin{equation}\label{eq. defn gen funct F}
\mathcal F_{\Gamma,s}(z,\tau)=\mathcal F_s(z,\tau):= 2\pi\sum_{n\ge1}\sqrt{n}F_{-n,s}(\tau)e^{2\pi inz}\qquad(z,\tau\in\H).
\end{equation}
If $\Gamma=\PSL_2(\Z)$, then $2\pi \sqrt{n}F_{-n,1}= j_n$. Here $j_1:=j-744$ with $j$ the classical $j$-invariant and $j_n=nj_1\big|T_n$, where $T_n$ is the normalized $n$-th Hecke operator. In \cite{AKN97}, $H_\tau(z)=\mathcal F_{\PSL_2(\Z),1}(z,\tau)+1$ was studied, where it was proved that (see Theorem 3 of \cite{AKN97})
$$
H_\tau(z)=-\frac{1}{2\pi i}\frac{j'(z)}{j(z)-j(\tau)}.
$$
This identity is equivalent to the denominator formula for the Monster Lie algebra. The fact that $z\mapsto H_\tau(z)$ is a weight two meromorphic modular form on $\PSL_2(\Z)$ was further utilized in the study divisors of modular forms in \cite{BKO04}. An extension of those results to $\Gamma_0(N)$ was derived in \cite{BK16,BKLOR18}; see also \cite{Lo18} where a regularized inner product of $H_\tau$ and $H_z$ was evaluated.

Moreover, we study the following generating function of derivatives in $s$ of $[\frac{\partial}{\partial s}\{F_{-n,s}\}_{n\geq 1}]_{s=1}$
\begin{equation}\label{eq. defn gen deriv F}
\mathbb{F}_{\Gamma}(z,\tau)=\mathbb{F}(z,\tau) :=\left[\frac{\partial}{\partial s}\mathcal F_{s}(z,\tau)\right]_{s=1}= 2\pi\sum_{n\ge1}\sqrt{n}\left[\frac{\partial}{\partial s}F_{-n,s}(\tau)\right]_{s=1}e^{2\pi inz}.
\end{equation}
If $\im(z) > \im(\gamma\tau)$ for all $\gamma \in \Gamma$, then \eqref{eq. defn gen funct F} and \eqref{eq. defn gen deriv F} converge locally uniformly.
In the sequel we omit $\Gamma$ from the notation if the group is fixed. We also write $\mathcal F(z,\tau) := \mathcal F_1(z,\tau)$.

%\subsection{Main results}

\subsection{The generating function $\mathcal F(z,\tau)$ as an automorphic form}

%{\bf KB: I don't think $\mathcal F$ without subscript is defined yet.}
For $z\in\mathbb{H}$ with $z\neq \tau$, we prove that the generating function $\mathcal F(z,\tau)$ is related to the harmonic form
\begin{equation}\label{G}
\mathcal{G}(z,\tau):=\lim_{s\to 1} \left(G_s(z,\tau) +E(\tau;s)\right)\!,
\end{equation}
where  $E(\tau;s)$ is the real-analytic (parabolic) Eisenstein series associated to the cusp $i\infty$, see Subsection \ref{sec. parab Eis}
(resp. Subsection \ref{sec. aut green}) for a discussion of the relevant aspects of the Eisenstein series (resp. automorphic Green's function).
Throughout, we write $w\in\C$ as $w=w_1+iw_2$ with $w_1, w_2\in\R$.
We denote the Bruinier--Funke {\it$\xi$-operator} by $\xi_{\kappa,z} := 2iz_2^\kappa\overline{\frac{\partial}{\partial\overline z}}$
and the {\it Maass raising operator in weight zero} by $R_{0,z}:=2i\frac{\partial}{\partial z}$.
If it is clear from the context with respect to which variable a certain operator is taken,
then we may drop this dependence from the notation. Polar harmonic Maass forms behave
like Maass forms with respect to the action by $\Gamma$.  However, they are allowed to have poles in
the upper half-plane, see Subsection \ref{S:Maass} for details. Letting $$\widehat{\mathcal F}(z,\tau):=\mathcal F(z,\tau)+\frac{1}{\vol(\Gamma\backslash\mathbb{H})z_2},$$ our first main result is the following.

\begin{theorem}\label{thm: main 1}\hspace{0cm}
	\begin{enumerate}[label=\rm(\arabic*),leftmargin=*]
		\item For $z,\tau\in\mathbb{H}$, with $z\neq \gamma \tau$ for all $\gamma\in\Gamma$, we have
		\begin{equation*} %\label{eq. gen series of F_n}
			\widehat{\mathcal{F}}= R_{0,z}(\mathcal{G}(z,\tau)).
		\end{equation*}
		\item The function $\tau \mapsto \mathcal F(z,\tau)$
is a weight zero polar harmonic Maass form.
		\item The function $z \mapsto \widehat{\mathcal F}(z,\tau)$
is a weight two polar harmonic Maass form with the only pole in the fundamental domain at $z=\tau$. The corresponding polar part equals
%{\bf KB: Simplified: $\frac{1}{z-\overline{\tau}}$ has no pole in $\H$.}
	\begin{equation*}%\label{eq. pricipal part}
%			-\frac{ \mathrm{Stab}_\tau\cdot \tau_2}{\pi(z-\tau)(z-\overline\tau)} =
  \frac{i\mathrm{Stab}_\tau}{2\pi(z-\tau)},
		\end{equation*}
where $\mathrm{Stab}_\tau$ denotes the order of the stabilizer group of the point $\tau$ in $\Gamma$.
	\end{enumerate}
\end{theorem}

Theorem \ref{thm: main 1} extends results from \cite{BK16,BKLOR18} as follows. Suppose that $\Gamma=\overline{\Gamma_0(N)} := \Gamma_0(N)/\{\pm I_2\}$ ($I_2:=\begin{psmallmatrix}
1&0\\0&1
\end{psmallmatrix}$) is the congruence group of level $N\in\mathbb{N}$.
In Theorem 1.1 of \cite{BKLOR18}, $z\mapsto \mathcal F_{\overline{\Gamma_0(N)}}(z,\tau)$ was studied. There it was proven that for $z_2>\max\{\tau_2,\tau_2^{-1}\}$, one has
\begin{equation}\label{eq. gen function different express}
\mathcal F_{\overline{\Gamma_0(N)}}(z,\tau)=H_{N,\tau}^{\ast}(z)-
\frac{3}{\pi\left[\mathrm{SL}_2(\mathbb{Z}):\Gamma_0(N)\right]z_2}
\end{equation}
for some function $H_{N,\tau}^{\ast}$, which was studied in detail in \cite{BK16}.
Specifically, it was shown that $H_{N,\tau}^{\ast}(z)$ is polar harmonic Maass form of weight two, if viewed as a function of $z$,
and of weight zero, if viewed as a function of $\tau$.
In \cite{BK16} the function $H_{N,\tau}^{\ast}$ was defined through analytic continuation to $s=0$ of the Poincar\'e series (defined for $\re(s)>0$)
$$
P_{N,s}(z,\tau) :=\sum_{\gamma\in\Gamma_0(N)}\frac{\varphi_s(\gamma z, \tau)}{j(\gamma,z)^2 |j(\gamma,z)|^{2s}}.
$$
Here, we let $j(\begin{psmallmatrix}a&b\\c&d\end{psmallmatrix}, z):=cz+d$ and
$$
\varphi_s(z,\tau):=\tau_2^{1+s}(z-\tau)^{-1}(z-\overline\tau)^{-1}|z-\overline\tau|^{-2s}.
$$
Theorem \ref{thm: main 1} holds for any Fuchsian group of the first kind which has a cusp $i\infty$ of width one
(even it not arithmetic).  As such, our result extends the above cited result from \cite{BKLOR18}, whose
methods are specific if $\Gamma$ is a congruence group.  Furthermore, note that
$$
\frac{3}{\pi[\mathrm{SL}_2(\mathbb{Z}):\Gamma_0(N)]}=
\frac1{\vol\left(\overline{\Gamma_0(N)}\backslash\PSL_2(\Z)\right)}.
$$
Combining this with Theorem \ref{thm: main 1} (1) and \eqref{eq. gen function different express}, Theorem \ref{thm: main 1} yields, in the notation of\footnote{In view of the construction of $H_{N,\tau}^{\ast}(z)$ in
	\cite{BK16} we find this identity rather surprising.} \cite{BK16} that\footnote{Our method of proof for Theorem \ref{thm: main 1} employs spectral theoretic techniques, which
	is the reason we are able to consider general Fuchsian groups.}
%{\bf KB: Here I would use the dependence on the group}
$$
H_{N,\tau}^{\ast}(z)=R_{0,z}\left(\mathcal{G}_{\overline{\Gamma_0(N)}}(z,\tau)\right).
$$

Theorem \ref{thm: main 1} allows us to study divisors of weight $k$ meromorphic modular forms on $\Gamma$, assuming that those are not supported at the cusps. Namely, for a weight $k$ meromorphic modular form on $\Gamma$ with divisors not supported at the cusps of $\Gamma$, define the \emph{divisor polar harmonic Maass form}
\begin{equation}\label{eq. div dom form}
f^{\rm div}(z):=\sum_{\tau\in\mathfrak F}\frac{\mathrm{ord}_\tau(f)}{\mathrm{Stab}_\tau}\widehat{\mathcal{F}}(z,\tau),
\end{equation}
where $\mathfrak F$ is the fundamental domain for $\Gamma \backslash \mathbb{H}$ and $\mathrm{ord}_\tau(f)$ denotes the order of $f$ at $\tau$. With this notation, we have the following corollary.
\begin{corollary}\label{cor. divisor}
  Let $f$ be a weight $k$ meromorphic modular form on $\Gamma$ with a divisor not supported at the cusps of $\Gamma$. Then, with $\Theta:=\frac{1}{2\pi i}\frac{\partial}{\partial z}$ and $S_2(\Gamma)$ the space of weight two cusp forms $\Gamma$,
  $$
  f^{\rm div}(z)\equiv \frac{k}{4\pi z_2}-\frac{\Theta(f(z))}{f(z)} \pmod{S_2(\Gamma)}.
  $$
\end{corollary}

Before we continue with our deeper study of these generating functions, we would like to emphasize that Theorem \ref{thm: main 1} and Corollary \ref{cor. divisor} hold true for any Fuchsian group with at least one cusp. This allows one to study arithmetic properties of automorphic forms on non-arithmetic groups, see e.g. \cite{Wo83} where non-arithmetic triangle groups were considered. Further arithmetic studies of non-arithmetic triangle groups include \cite{DGMS13, MS14, Tr19}; see also \cite{MU23} for an application of harmonic Maass forms on triangle groups to low dimensional topology. Non-arithmetic groups of interest include the semi-arithmetic Fuchsian groups introduced in \cite{SW00}; see \cite{MZ16} for a thorough study of automorphic forms on those groups.
% where a classification of all primes appearing in the denominators of the Hauptmodul and modular forms for non-arithmetic triangle groups with a cusp was given.
%A definition of a "CM" point for cuspidal triangle groups as a point corresponding to an abelian variety with complex multiplication was given in \cite{Tr19} where it was proved that zeros of automorphic forms satisfying certain natural assumption, are either transcendental or CM. }

Corollary \ref{cor. divisor} is particularly useful if genus of $\Gamma\backslash\mathbb{H}$ is
zero, and this is the case for non-arithmetic triangle groups with a cusp. In this case one also has that the Hauptmodul $J$ for the group $\Gamma$ is such that $2\pi \sqrt{n}F_{-n,1}(\tau) = T_n(J(\tau))+ c_n$, where $T_n$ is the $n$-th Faber polynomial and $c_n$ is a certain explicitly computable constant depending only on $n$ and $\Gamma$. Then Corollary \ref{cor. divisor} becomes 
\begin{equation*}
\frac{\Theta(f(z))}{f(z)}= - \sum_{\tau\in\mathfrak F}\frac{\mathrm{ord}_\tau(f)}{\mathrm{Stab}_\tau}\mathcal{F}(z,\tau)=\sum_{n\ge1} 2\pi\sqrt{n} \sum_{\tau\in\mathfrak F}\frac{\mathrm{ord}_\tau(f)}{\mathrm{Stab}_\tau} F_{-n,1}(\tau) e^{2\pi i n z},
\end{equation*}
which is particularly useful for the study of arithmetic properties of forms on genus zero non-arithmetic groups. We plan to pursue this study in forthcoming work.

\subsection{The generating function $\mathcal F(z,\tau)$ as a series over $\Gamma_{\infty}\setminus\Gamma$}

Since \eqref{eq. defn gen funct F} is invariant with respect to $\Gamma_{\infty}$, both in $z$ and in $\tau$, it is natural to ask if one can write \eqref{eq. defn gen funct F} in certain cases as a series over $\Gamma_{\infty}\setminus\Gamma$.  We prove that
one can do so if $s=n \in \mathbb{N}$.

\begin{theorem}\label{thm: main 2_1}\hspace{0cm} With notation as above,
for $z,\tau\in\mathbb{H}$ with $z_2>\im(\gamma \tau)$, we have\footnote{Throughout we denote by $\Log$ the principal branch of the logarithm.}, for all $\gamma\in\Gamma$,
$$
\mathcal F_{2}(z,\tau) = \sum_{\gamma\in\Gamma_\infty\setminus\Gamma}
\left(
\frac{e^{2\pi i\left(z-{\gamma\tau}\right)}}{1-e^{2\pi i\left(z-{\gamma\tau}\right)}} +
\frac{e^{2\pi i\left(z-\overline{\gamma\tau}\right)}}{1-e^{2\pi i\left(z-\overline{\gamma\tau}\right)}}
+
\frac{\Log\!\left(1-e^{2\pi i\left(z-\overline{\gamma\tau}\right)}\right)
-\Log\!\left(1-e^{2\pi i\left(z-\gamma\tau\right)}\right)}
{\pi i( \overline{\gamma\tau}-\gamma \tau)}\right)\!.
$$
\end{theorem}

\noindent
Additionally,
we describe how one can similarly write $\mathcal F_{n} (z,\tau)$ for $n\ge 2$ as a series over $\Gamma_{\infty}\setminus\Gamma$
involving polylogarithms. In the case $s=1$, we prove the following result.

\begin{theorem}\label{thm: main 2_2}\hspace{0cm} Using notation as above,
for $z, \tau\in\mathbb{H}$ with $z_2>\im(\gamma \tau)$ for all $\gamma\in\Gamma$, we have
$$
 \sum_{\genfrac{}{}{0pt}{}{\gamma\in\Gamma_\infty\setminus\Gamma}{\im(\gamma\tau) > \varepsilon}}
\left(\frac{1}{1-e^{2\pi i (z-\gamma\tau)}} - \frac{1}{1-e^{2\pi i (z-\overline{\gamma\tau})}}\right)
=\mathcal F_{1} (z,\tau) + O\left(\frac{1}{\log (\varepsilon)}\right)
\,\,\,\,\,
\text{\rm as $\varepsilon \rightarrow 0^{+}$.}
$$
In particular,
$$
\mathcal F_{1} (z,\tau) = \lim\limits_{\varepsilon \rightarrow 0}
\sum_{\genfrac{}{}{0pt}{}{\gamma\in\Gamma_\infty\setminus\Gamma}{\im(\gamma\tau) > \varepsilon}}
\left(\frac{1}{1-e^{2\pi i (z-\gamma\tau)}} - \frac{1}{1-e^{2\pi i (z-\overline{\gamma\tau})}}\right).
$$
\end{theorem}

\subsection{The generating function $\mathbb{F}(z,\tau)$}

In Proposition \ref{prop. holom part}, we prove that $z\mapsto\mathbb{F}(z,\tau)$ is the holomorphic part of the biharmonic Maass form
$
	[\frac{\partial}{\partial s}\xi_{0,z}(G_{\bar s}(z,\tau))]_{s=1}.
$
Recall that \emph{biharmonic in weight two} means that the function
is annihilated by $\Delta_{2,z}^2=\left(\xi_{0,z}\circ\xi_{2,z}\right)^2$.
We prove
an explicit representation for $\mathbb{F}(z,\tau)$ in terms of the raising operator.
Specifically, we show:

%{\bf KB: Maybe the idea of proof is not needed here? I think one can only follow it if one reads it below. - I agree, commented the idea of proof explanation}
% In order to derive it, we start with the fact that $ \left[\frac{\partial}{\partial s}F_{-n,s}(\tau)\right]_{s=1}$ can be written in terms of certain regularized inner products defined in Subsection \ref{sec. regul inner prod} as
%$$
%\left[\frac{\partial}{\partial s}F_{-n,s}(\tau)\right]_{s=1} = -\left\langle %F_{-n,1}(\tau), \lim_{s\to 1}(G_s(\cdot,\tau)+E(\tau;s)) \right\rangle.
%$$
%In Proposition \ref{prop. interchange sum and int}, we prove that
%$$
%\mathbb{F}(z,\tau)=-\left\langle R_{0,z}(\mathcal{G}(z,\cdot)) - %\frac1{\operatorname{vol}(\Gamma\backslash \mathbb{H})z_2},  %\mathcal{G}(\cdot,\tau)\right\rangle.
%$$
%We then compute this inner product explicitly in the following theorem.

\begin{theorem}\label{thm. gen series of deriv}
 Assume that $z,\tau \in \mathbb{H}$ and $z\neq \tau$.
Let $L$ denote the Rogers dilogarithm function, defined in \eqref{E:RD}.  Set $u(z,\tau):=\frac{|z-\tau|^2}{4z_2\tau_2}$.  Define
$\beta$ and $P(\tau)$ as in \eqref{eq. KLF infty}, which corresponds to the constant term in the Laurent series expansion
\eqref{KronLimitPArGen} at $s=1$ of the Eisenstein series $E(\tau;s)$.  Then
  $$
  \mathbb{F}(z,\tau)= -R_{0,z}(\mathcal{K}(z,\tau))+\frac{\beta}{z_2}-\frac{P(\tau)}{\operatorname{vol}(\Gamma\backslash \mathbb{H})z_2},
  $$
  where
  \begin{equation}\label{eq:point_pair_series}
  \mathcal{K}(z,\tau):=\sum_{\gamma\in\Gamma} k(u(z,\gamma\tau))
  \end{equation}
  is the automorphic kernel associated to the point pair invariant
  \begin{equation*}%\label{eq. point pair}
  k(u):=\frac{1}{2\pi}\left(L(-u)+\frac{\pi^2}{6}\right).
  \end{equation*}

\end{theorem}

\begin{remark}\rm
{\it By Appendix A of \cite{za07}, $k$ is a real-analytic function (except at zero and one) satisfying $\lim_{u\to\infty}k(u)=0$. This
is necessary for \eqref{eq:point_pair_series} to converge.
If $\Gamma=\SL_2(\Z)$, then one could alternatively prove \Cref{thm. gen series of deriv} using modular properties and
the growth behavior of the specific functions involved, including the discriminant function.
However, this approach becomes somewhat complicated for general groups because the corresponding Kronecker limit functions are less understood.
Thus, we employ spectral theoretic methods in all cases to prove Theorem \ref{thm. gen series of deriv}.}
\end{remark}

%\textcolor{red}{Check the normalizing constant, as we use different normalizations than in \cite{Ar15} - the point-pair invariant is correct modulo a constant factor in front.}

\subsection{Organization of article}
This paper is organized as follows.  In Section \ref{sec:notation}, we establish notation and recall various known
results which are used in our proofs. In Section \ref{sec:Greens_comps}, we obtain certain asymptotic results for automorphic Green's functions in cuspidal regions, which are needed in our proofs of Theorem \ref{thm: main 1}
and Theorem \ref{thm. gen series of deriv}. In Section \ref{sec:Proof_of_Thm1}, we prove Theorem \ref{thm: main 1}, and
in Section \ref{sec:Proof_of_series}, we show Theorem \ref{thm: main 2_1} and Theorem \ref{thm: main 2_2}. In
Section \ref{sec:Inner_Products}, we show that $\mathbb{F}(z,\tau)$ can be realized as a regularized inner product involving
the functions defined in Theorem \ref{thm: main 1} and \ref{G}.  The proof of Theorem \ref{thm. gen series of deriv}
is given in Section \ref{sec:Proof_of_gen_series_deriv}.

\section*{Acknowledgements}
The first author has received funding from the European Research Council (ERC) under the European Union’s Horizon 2020
research and innovation programme (grant agreement No. 101001179). The second author acknowledges grant support
from PSC-CUNY Award 67415-00-55, which was jointly funded by the Professional Staff Congress and The City University of New York.
The authors thank CIRM where we met during the Building Bridges 6 (BB6) conference in September 2024. There the authors began developing the results presented in this article.

\section{Preliminaries}\label{sec:notation}

\subsection{The Rogers dilogarithm}\label{sec. rogers dilog}

For $w\in\C$ with $|w|<1$ and $s\in\C$, the {\it polylogarithm} is
\begin{equation}\label{pol}
	\Li_{s}(w) := \sum_{n\ge1} \frac{w^n}{n^{s}}.
\end{equation}
Note that
$$
\Li_{0}(w) = \frac{w}{1-w}
\,\,\,\,\,
\text{\rm and}
\,\,\,\,\,
\Li_{1}(w) =-\Log(1-w).
$$
If $s = 2$, then one calls $\Li_2(w)$ the {\it dilogarithm}.  One can prove that
a single-valued analytic continuation of $\Li_2(w)$ exists to the split plane $\C\setminus[1,\infty)$.
Furthermore, one has
(see 25.12.3 of \cite{DLMF})
\begin{equation*}
	\Li_2 (w)+\Li_2 \left(\frac{w}{w-1}\right) = -\frac{1}{2} \Log^{2}(1-w)
\,\,\,\,\text{\rm for} \,\,\,\, w\in \C \setminus[1,\infty).
\end{equation*}
We let, for $x\in(0,1)$,
\begin{equation}\label{E:RD}
	L(x):=\mathrm{Li}_2(x)+\frac{1}{2}\log(x)\log(1-x).
\end{equation}
It can be extended to $\mathbb{R}$ by setting $L(0):=0$, $L(1):=\frac{\pi^2}6$, and
$$
L(x):=
\begin{cases}
	2L(1)-L(x) & \text{if $x>1$}, \\
	-L\left(\frac x{x-1}\right) & \text{if $x<0$},
\end{cases}
$$
see \cite[Section A]{za07} for these and other basic facts. A direct calculation shows the following lemma.

\begin{lemma}\label{lem. dilog}
	For $u\in\R^+$, we have
	$$
	L(-u)=\mathrm{Li}_2\left(-u\right)+\frac{1}{2}\log(u) \log(u+1).
	$$
\end{lemma}

\subsection{Basic notation}\hspace{0cm}\label{sec. Basic}
Let $\Gamma\subset\PSL_2(\R)$ be a Fuchsian
group of the first kind acting by fractional
linear transformations on the upper half-plane $\mathbb{H}$.
Set $M:=\Gamma\backslash\mathbb{H}$, which is a finite
volume hyperbolic Riemann surface.  We let $ds^2:=\frac{dz_1^2+dz_2^2}{z_2^2}$ be the {\it hyperbolic line element} in the coordinate
$z=z_1+iz_2\in\mathbb{H}$.  The line element induces a metric, called the {\it hyperbolic metric}, on $M$ which is compatible with its complex structure.
The hyperbolic metric has constant negative curvature $-1$.
Let $d\mu= \mu(z):=\frac{dz_1 dz_2}{z_2^2}$ denote the associated {\it hyperbolic volume element} on $M$. The Riemann surface $M$ may have elliptic fixed points.  We assume that $M$ has at least a cusp, say $i\infty$, which is assumed to have width one.
Slightly abusing notation, we also say that $\Gamma$ has a cusp $i\infty$ of width one, meaning that the stabilizer $\Gamma_\infty$ of $i\infty$
is generated by the matrix $\bigl(\begin{smallmatrix}
1&1\\0&1\end{smallmatrix}\bigr)$. We identify $M$
locally with its universal cover $\mathbb{H}$. We denote by $\mathfrak F=\mathfrak F(\Gamma)$ the ``usual'' (Ford) fundamental domain for $\Gamma$ acting
on $\mathbb H$.
We assume that $\Gamma$ has $K\geq 0$ cusps other than $i\infty$, which we denote, for $\ell\in\{1,\ldots,K\}$, by $\varrho_\ell$.  Let $\Gamma_{\varrho_\ell}:=\{\gamma\in\Gamma : \gamma\varrho_\ell=\varrho_\ell\} =\langle\gamma_{\varrho_\ell}\rangle$ denote the \emph{stabilizer group} of $\varrho_\ell$. The matrices $\sigma_\ell$ such that for $\ell\in\{1,\ldots, K\}$ satisfying $\sigma_\ell i\infty=\varrho_\ell$ and $\sigma_\ell^{-1}\gamma_\ell\sigma_\ell= \bigl(\begin{smallmatrix}
1&1\\0&1\end{smallmatrix}\bigr)$ are the {\it scaling matrices} of the cusps $\varrho_\ell$.

\subsection{Harmonic Maass forms and polar harmonic Maass forms}\label{S:Maass}
{\it Harmonic Maass forms} of weight $k\in \mathbb{Z}$ for $\Gamma$ are real-analytic functions that transform modular of weight $k$ on $\Gamma$, have at most linear exponential growth at the cusps, and are annihilated by
\begin{equation*}
\Delta_{k,\tau} := -\tau_2^2\left(\frac{\partial^2}{\partial\tau_1^2}+\frac{\partial^2}{\partial\tau_2^2}\right) + ik\tau_2\left(\frac\partial{\partial\tau_1}+i\frac\partial{\partial\tau_2}\right),
\end{equation*}
the {\it weight $k$ hyperbolic Laplace operator}. Note that $\Delta_{k,\tau}=-\xi_{2-k,\tau}\circ \xi_{k,\tau}$. {\it Polar harmonic Maass forms} may have poles in the upper half-plane and, for all $z\in\H$, there exists $n\in\N_0$ such that $(\tau-z)^nf(\tau)$ is bounded in some neighborhood of $z$. The {\it polar part}  of $f$ in $z$ is given via
\begin{equation*}
	f(\tau)=\sum_{j=0}^{n_0}\frac{a_j(\overline \tau)}{(\tau-z)^j}+f^{*}(\tau)
\end{equation*}
where $f^{*}$ is bounded in a neighbourhood of $z$. Moreover, $f$ is allowed to have most linear exponential growth at $i\infty$. Analogue conditions are required at the other cusps of $\Gamma$.

\subsection{Niebur--Poincar\'e series}
For $m$ and $\tau$ fixed, \eqref{eq. N-P series defn} converges absolutely and uniformly on any compact subset of the half-plane
$\re(s)>1$. Moreover, $\Delta_0(F_{m,s}) = s(1-s) F_{m,s}$. From Theorem 5 of \cite{Ni73}, the function $F_{m,s}$ has a meromorphic continuation to the whole complex $s$-plane.
It is holomorphic at $s=1$ which follows from its Fourier expansion.  Namely, for $\re(s)>1$, we have
%{\bf KB: In this subsection there was a mess-up between $\tau$ and $z$. Made consistent.}
\begin{equation*}%\label{Four exp Nieb}
F_{m,s}(\tau)=\sqrt{\tau_2}I_{s-\frac12}(2\pi |m|\tau_2)e^{2\pi im \tau_1} + \sum_{n\in\Z}b_{m,s}(\tau_2;n)e^{2\pi i n \tau_1},
\end{equation*}
where\footnote{Note that we correct some typos in Theorem 1 of \cite{Ni73}.}
$$
b_{m,s}(\tau_2;n)= B_{m,s}(n)\begin{cases}
	\frac{\tau_2^{1-s}}{2s-1}&\text{if }n=0,\\
	\sqrt{\tau_2}K_{s-\frac12}(2\pi |n|\tau_2)&\text{if }n\neq0.
\end{cases}
$$
Here, $I_{\kappa}$ and $K_{\kappa}$ are the modified Bessel functions of order $\kappa$.  Furthermore,
\begin{align*}
	B_{m,s}(0)&:= \frac{2\pi^s}{\Gamma(s)}|m|^{s-\frac{1}{2}}\sum_{c\ge1}\frac{S(m,0;c)}{c^{2s}},\\
	B_{m,s}(n)&:= 2 \sum_{c\ge1}\frac{S(m,n;c)}c
	\begin{cases}
		J_{2s-1}\left(\frac{4\pi}{c} \sqrt{mn}\right) & \textrm{\rm if \,}mn>0, \\[-0.3cm]\\
		I_{2s-1}\left(\frac{4\pi}{c} \sqrt{|mn|}\right) & \textrm{\rm if \,} mn<0,
	\end{cases}
\end{align*}
 where $J_{\kappa}$ denotes the $J$-Bessel function of order $\kappa$ and $S(m,n;c)$ is the Kloosterman sum
$$
S(m,n;c):= \sum_{\bigl(\begin{smallmatrix}
a&\ast\\c&d\end{smallmatrix}\bigr)\in \Gamma_\infty\backslash\Gamma/\Gamma_\infty }e^{\frac{2\pi i}{c}(ma+nd)}.
$$
Using that
\begin{equation}\label{eq:Bessel_identities1}
I_\frac12(x)=\sqrt\frac{2}{\pi x} \sinh(x)
\,\,\,\,\,
\text{\rm and}
\,\,\,\,\,
K_\frac12(x)=\sqrt\frac\pi{2x} e^{-x},
\end{equation}
we have that
%{\bf KB: Fixed typos}
\begin{equation*}%\label{Four exp Nieb at 1}
F_{m,1}(\tau)=\frac{\sinh(2\pi|m|\tau_2)}{\pi \sqrt{|m|}}e^{2\pi im\tau_1} + B_{m,1}(0)+ \frac12 \sum_{n\in\Z\setminus\{0\}} \frac{B_{m,1}(n)}{\sqrt{|n|}} e^{-2\pi|n|\tau_2+2\pi in\tau_1}.
\end{equation*}
From this, it follows that $F_{m,1}$ is annihilated by $\Delta_0$. Moreover, we have
\begin{equation}\label{eq. NP deriv bound0}
	\left|F_{-m,1}(\tau) \right| \ll q^{-m} \quad\text{as }\tau_2\to\infty.
\end{equation}
Similarly, using 10.38.6, 6,12,1, and 6.12.2 of \cite{DLMF}, one can show that
%{\bf KB: This is only true if $\ll$ depends on $\tau_1$. I would write it as $e^{-2\pi in\tau}$ instead.}\textcolor{blue}{LS: I agree. changed.} {\bf KB: I get $e^{-2\pi in\tau}$.}
\begin{equation}\label{eq. NP deriv bound}
\left[\frac{\partial}{\partial s}F_{-m,s}(\tau)\right]_{s=1} \ll \frac{q^{-m}}{\tau_2} \quad \text{as $\tau_2\to\infty$}.
\end{equation}

We next turn to the Niebur--Poincar\'e series associated to the other cusps. Let $K$ denote the number of cusps of $\Gamma$ different from $i\infty$. If $K>0$, then, for $\ell\in\{1,\ldots K\}$, the {\it Niebur--Poincar\'e series associated to} $\varrho_\ell$ is defined, for $\re(s)>1$, as absolutely convergent series (see e.g. Definition 3.2 of \cite{Ne73})
\begin{equation}\label{eq. NP series gen cusp}
F_{m,s,\varrho_\ell}(\tau):= \sum_{\gamma\in\Gamma_{\varrho_\ell} \backslash \Gamma} \sqrt{\im\left(\sigma_\ell^{-1}\gamma \tau\right)}I_{s-\frac12}\left(2\pi |m| \im\left(\sigma_\ell^{-1}\gamma \tau\right)\right) e^{2\pi im \re\left(\sigma_{\ell}^{-1}\gamma \tau\right)}.
\end{equation}
The function $F_{m,s,\varrho_\ell}$ has a meromorphic continuation to the entire $s$-plane which is holomorphic at $s=1$ and is a real-analytic eigenfunction of $\Delta_0$ with eigenvalue $s(1-s)$. If $K\geq 1$, then, by Lemma 4.2 of \cite{Ne73}, if $\tau$ is sufficiently close to the cusp $\varrho_\ell$ (i.e., $\im(\sigma_\ell^{-1} \tau)\gg1$)
and $\re(s)\geq 1$, then
\begin{equation}\label{eq. bound for NP at other cusps}
|F_{-m,s}(\tau)|\leq C(2\pi m)^{\re(s)-1}e^{2\pi m}\im\left(\sigma_\ell^{-1} \tau\right)^{1-\re(s)},
\end{equation}
where, the constant $C$ depends solely on $\Gamma$.

Let $P(Y):=\{w=w_1+iw_1: 0<w_1<1,\, w_2\geq Y\}$ denote a \emph{horocycle} on $M$. By considering
the Fourier expansion of $F_{-m,1,s}(\tau)$ at the cusp $\varrho_\ell$, given by (4.29) in \cite{Ne73}, one obtains\footnote{Recall that $\frac{\partial}{\partial\mathrm n}g(\sigma_{\ell}\tau):=[\frac{\partial}{\partial \tau_2}g(\sigma_{\ell}\tau)]_{\tau_2=Y}$ if $\tau=\tau_1+iY \in P(Y)$, see e.g. proof of Theorem 6.14 in \cite{Iwa02}.}
\begin{equation}\label{eq. bound for norm der NP at other cusps}
\bigg|\frac{\partial}{\partial\mathrm n}F_{-m,1}(\tau)\bigg|\leq \frac{Ce^{2\pi m}}{\sqrt{2\pi m}Y}
\,\,\text{\rm for} \,\, \tau\in\sigma_\ell P(Y) \,\, \text{\rm as}\,\, Y:=\im\left(\sigma_\ell^{-1} \tau\right)\to \infty.
\end{equation}
By using the Fourier expansions of \eqref{eq. NP series gen cusp} with $m\in-\mathbb N$ and by arguing in the same way as
in the proof of Lemma 4.2 of \cite{Ne73} and Lemma 3 of \cite{CJS23}, one can show, as $\im(\sigma_\ell^{-1}\tau)\to \infty$,
\begin{equation}\label{eq. NP growth at other cusps}
F_{-m,1,\varrho_\ell}(\tau)\ll e^{2\pi m \im\left(\sigma_\ell^{-1}\tau\right)} \quad \text{and} \quad \left[\frac{\partial}{\partial s}F_{-m,s,\varrho_\ell}(\tau)\right]_{s=1}\ll \frac{e^{2\pi m \im\left(\sigma_\ell^{-1}\tau\right)}}{\sqrt{\im \left(\sigma_\ell^{-1}\tau\right)}}.
\end{equation}

\subsection{Parabolic Eisenstein series}\label{sec. parab Eis}

The {\it non-holomorphic, parabolic Eisenstein series} associated to the cusp $i\infty$ is defined, for $s\in\C$ with $\re(s)>1$, by
\begin{equation}\label{eq:para_eisen}
E(\tau;s):=\sum_{\gamma\in\Gamma_\infty\setminus\Gamma}\im(\gamma \tau)^s.
\end{equation}
It has a meromorphic continuation to the entire $s$-plane with a simple pole at $s=1$.  As $s\to1$, $E(\tau;s)$ has the Laurent expansion given
by\footnote{This
 follows from Theorem 3.1 of \cite{Go73}, after trivial renormalization and generalization (see also Subsection 3.1. of \cite{JO'SS20} where
 minor inconsistencies in \cite{Go73} were rectified).}
\begin{equation} \label{KronLimitPArGen}
	E(\tau;s)= \frac{1}{\vol(\Gamma\backslash\H) (s-1)} + \beta- \frac{\log\left(\left|\eta_{i\infty}(\tau)\right|^4 \tau_2\right)}{\vol(\Gamma\backslash\H)}  + P_1(\tau)(s-1)+ O\left((s-1)^2\right),
\end{equation}
where $\beta=\beta_{\Gamma}$ is a certain real constant depending only on the group $\Gamma$ and $\eta_{i\infty}$ is a certain function transforming modular of weight $\frac{1}{2}$ for $\Gamma$. The function $\tau\mapsto E(\tau;s)$ is an eigenfunction of $\Delta_0$ with eigenvalue $s(1-s)$. The {\it parabolic Kronecker limit function} for $\Gamma$ is defined as
\begin{equation} \label{eq. KLF infty}
P(\tau) := \log\left(\left|\eta_{i\infty}(\tau)\right|^4 \tau_2\right).
\end{equation}
If $K\geq 1$, then, for $\ell\in\{1,\ldots,K\}$, the {\it Eisenstein series associated to the cusp} $\varrho_\ell$ is defined,
$$
E_{\varrho_\ell}(\tau;s) := \sum_{\gamma\in\Gamma_{\varrho_\ell}\setminus\Gamma}\im\left(\sigma_\ell^{-1}\gamma \tau\right)^s \qquad (s\in\C,~ \re(s)>1).
$$
It has a meromorphic continuation to the complex $s$-plane with a simple pole at $s=1$ with residue  $\vol(\Gamma\backslash\mathbb{H})^{-1}$.
Moreover, the analysis used in Subsection 3.1 of \cite{JO'SS20} yields that for any pair of cusps $\varrho_\ell, \varrho_j$
($j,\ell\in\{0,1,\ldots, K\}$, $\varrho_0:=i\infty$, $\sigma_0 := \mathrm{I_2}$) there exists a constant
$\beta_{\ell,j}=\beta_{\ell,j}(\Gamma)$, depending only on $\Gamma$ and a holomorphic form $\eta_{\varrho_{\ell},\varrho_j}$, such that, as $s\rightarrow 1$,
\begin{equation} \label{eq. KronLimitPAr cusps}
	E_{\varrho_\ell}(\sigma_j\tau;s)
	= \frac{1}{\vol(\Gamma\backslash\H)(s-1)} + \beta_{\ell,j}- \frac{P_{\varrho_{\ell},\varrho_j}(\tau)}{\vol(\Gamma\backslash\H)}
 + P_{\ell,j,1}(\tau)(s-1)+ O\left((s-1)^2\right),
\end{equation}
where the {\it parabolic Kronecker limit function} of $\Gamma$ \textit{at the pair of cusps $\varrho_\ell, \varrho_j$} is defined as\footnote{Note that $P(\tau)= P_{i\infty,i\infty}(\tau)$.}
$$
P_{\varrho_{\ell},\varrho_j}(\tau) := \log\left(\left|\eta_{\varrho_{\ell},\varrho_j}(\tau)\right|^4 \tau_2\right).
$$

The Fourier expansion of the Eisenstein series $E_{\varrho_\ell}(\sigma _j \tau;s)$
at a pair of cusps $\varrho_\ell, \varrho_j$ ($j,\ell\in\{0,\ldots,K\}$) is, by Theorem 3.4 of \cite{Iwa02}, for $\re(s)>1$, given by
$$
E_{\varrho_\ell}(\sigma _j \tau;s)=\delta_{\ell=j}\tau_2^s + \varphi_{\ell, j}(s)\tau_2^{1-s}+ 2\sqrt{\tau_2} \sum_{n\in\Z\setminus\{0\}}\varphi_{\ell, j}(s,n)K_{s-\frac12}(2\pi|n|\tau_2)e^{2\pi i n \tau_1}.
$$
Here we employ the notation $\delta_{\mathcal S}:=1$ if a statement $\mathcal S$ is true and $0$ otherwise.  Furthermore,
$$
\varphi_{\ell, j}(s):=\sqrt{\pi}\frac{\Gamma\!\left(s-\frac12\right)}{\Gamma(s)}\sum_{c\ge1}\frac{S_{\ell, j}(0,0;c)}{c^{2s}}
\,\,\,\,\text{\rm and}\,\,\,\,
\varphi_{\ell, j}(s,n)=\pi^s\frac{|n|^{s-1}}{\Gamma(s)}\sum_{c\ge1}\frac{S_{\ell, j}(0,n;c)}{c^{2s}},
$$
where the Kloosterman sums associated to the cusps $\varrho_\ell, \varrho_j$ ($j,\ell\in\{0,\ldots,K\}$) are
$$
S_{\ell, j}(m,n;c) := \sum_{\bigl(\begin{smallmatrix}
a&\ast\\c&d\end{smallmatrix}\bigr)\in \Gamma_\infty\backslash\sigma_\ell^{-1}\Gamma\sigma_j/\Gamma_\infty }e^{\frac{2\pi i}{c}(ma+nd)}.
$$
Note that $S_{0,0}(m,n;c)=S(m,n;c)$.
From the Fourier expansion of $E_{\varrho_\ell}(\sigma _j \tau;s)$ one can deduce that $P_{\varrho_{\ell},\varrho_j}(\tau)$ and $P_{1,\ell,j}(\tau)$ ($j,\ell\in\{0,\ldots,K\}$) have at most
linear growth\footnote{The growth is even at most logarithmic (see (3.2)--(3.4) on p. 13 of \cite{JO'SS20}), but for our purposes subexponential growth suffices.} in $\tau_2$, as $\tau_2\to \infty$.

%{\bf A proof of the Kronecker limit function construction; to be moved outside later - I put it here bcs I need Fourier expansions above {\color{red}{Leja, can you do this? LS: You mean, move this out? I sure can if that's what you propose}} KB: Yes, referring to the bold comment that this should be moved.}

\subsection{Automorphic Green's function} \label{sec. aut green}
Define the {\it automorphic Green's function}\footnote{The automorphic Green function from \cite{CJS23} follows \cite{He83}, p.31 and it differs from the Green's function in Chapter 5 of \cite{Iwa02} and from the function $G_1(z,\tau,s)$ from \cite[Section 5]{Ne73} by a minus sign. A different normalisation was also used in \cite{BK}, where the Green's function equals $-4\pi G_s(z,\tau)$ or in \cite{BKvP19}, where the Green's function equals $4\pi G_s(z,\tau)$.}
\begin{equation}\label{eq. defn Green}
G_s(z,\tau):=\sum_{\gamma\in\Gamma} k_s(\gamma z,\tau)\quad(z, \tau\in\H, \,s\in\C,\,\re(s)>1)
\end{equation}
where
\begin{equation}\label{ks}
	k_s(z,\tau) := -\frac{\Gamma^2(s)}{4\pi\Gamma(2s)} \left(1-\frac{|z-\tau|^2}{|z-\overline\tau|^2}\right)^s {}_2F_1\left(s,s;2s;1-\frac{|z-\tau|^2}{|z-\overline\tau|^2}\right).
\end{equation}
Here ${}_2F_1$ is the usual hypergeometric function. The Green's function is an eigenfunction of $\Delta_0$ in $z$ and $\tau$
with eigenvalue $s(1-s)$. Also, \eqref{eq. defn Green} is well-defined for $z\in \mathbb{H}$ satisfying $z\neq \gamma \tau$ for all $\gamma\in\Gamma$ and
$G_s(z,\tau)$ has a logarithmic singularity as $z\to \tau$  given by (see\footnote{Note the different sign, due to our different normalization.} \cite[Chapter 5]{Iwa02})

\begin{equation}\label{eq. greens function log singul expr}
G_s(z,\tau)=\frac{\mathrm{Stab}_\tau}{2\pi}\log|z-\tau|+ O(1).
\end{equation}
%{\bf KB: I feel it is enough to only have the Green's function as we use it.} In \cite{BK}, the {\it Green's function} on $\Gamma = \PSL_2(\Z)$ was defined as
%\begin{equation}\label{E:Alt}%
%	G_s(z,\mathfrak z) := \sum_{\gamma\in\SL_2(\Z)} g_s(z,\gamma\mathfrak z),
%\end{equation}
%where
%\begin{equation*}%
%	g_s(z,\mathfrak z) := 2Q_{s-1}\left(1+\frac{|z-\mathfrak z|^2}{2z_2\mathfrak z_2}\right).
%\end{equation*}
%Here $Q_s$ is the Legendre function of the second kind. In \cite[Subsection 2.6]{BKvP19}, it was noted that
%\begin{equation*}
%	g_s(z,\mathfrak z) = \frac{2^{s}\Gamma^2(s)}{\Gamma(2s)\cosh^s(d(z,\mathfrak z))} {}_2F_1\left(\frac s2,\frac{s+1}2;s+\frac12;\frac1{\cosh^2(d(z,\mathfrak z))}\right),
%\end{equation*}
%where we have
%\begin{equation*}%
%	\cosh(d(z,\mathfrak z)) = 1 + \frac{|z-\mathfrak z|^2}{2z_2\mathfrak z_2}.
%\end{equation*}
%We note that
%$$%
%	g_s(z,\tau)=-4\pi k_s(z,\tau).
%$$

\subsection{A Tauberian theorem}\label{subsection:Tauberian}

We rewrite \eqref{eq:para_eisen} as Dirichlet series
$$
E(z;s)=\sum_{\gamma\in\Gamma_\infty\setminus\Gamma}\left(\frac{1}{\im(\gamma z)}\right)^{-s}.
$$
One can apply general Karmata Tauberian theorems
to conclude that
\begin{equation*}%\label{eq:counting1_Eisen}
\sum_{\genfrac{}{}{0pt}{}{\gamma\in\Gamma_\infty\setminus\Gamma}{\im(\gamma z) > \varepsilon}} 1 =
\frac{\vol(\Gamma\backslash\H)}{\varepsilon} + o\left(\varepsilon^{-1}\right)
\,\,\,\,\,
\text{\rm as $\varepsilon \rightarrow 0$},
\end{equation*}
see, for example, page 27 of \cite{JL94} and references therein. On page 28 of \cite{JL94} it was shown that
by indexing $\frac{1}{\im(\gamma z)}=\{\lambda_{n}\}$, we have $\lambda_{n} \sim \frac{n}{\vol(\Gamma\backslash \H)}$,
as $n \rightarrow \infty$. Thus, the two theta functions
$$
\theta_{1}(t):=\sum_{\im(\gamma z)} e^{-\frac{t}{\im(\gamma z)}}
\,\,\,\,\,
\text{\rm and}
\,\,\,\,\,
\theta_{2}(t):=\sum_{n\ge1} e^{-\frac{nt}{\vol(\Gamma\backslash\H)}}
$$
have the same asymptotic main term as $t \rightarrow 0^{+}$.
As $\theta_{2}$ is a geometric series, it can be evaluated and, in particular, its
asymptotic behavior as $t\rightarrow 0^{+}$ is known.  With this, Corollary 2.4 of \cite{SS13} applies
so that the error term can be obtained from the analogous error term from the
counting function associated to $\theta_{2}$.  Specifically, one gets that
\begin{equation}\label{eq:counting2_Eisen}
\sum_{\genfrac{}{}{0pt}{}{\gamma\in\Gamma_\infty\setminus\Gamma}{\im(\gamma z) > \varepsilon}} 1 =
\frac{\vol(\Gamma\backslash\H)}{\varepsilon} + O\left(\frac{1}{\varepsilon \log (\varepsilon)}\right)
\,\,\,\,\,
\text{\rm as $\varepsilon \rightarrow 0$},
\end{equation}
see, for example, page 95 of \cite{BGV04}.  So, as discussed in the beginning of Chapter VII
of \cite{Korevaar}, one obtains an error term which, in general, is optimal.

\section{Computations involving Green's functions}\label{sec:Greens_comps}

\subsection{Certain bounds in cuspidal regions}
The Fourier expansion of $G_s(z,\tau)$ is described in the following lemma\footnote{Note that $G_s(z,\tau)=-\frac{\Gamma^2(s)}{8\pi\Gamma(2s)}G(\tau,z,s)$ in the notation of \cite{Ne73}.},
which is an immediate consequence of (4.56) and Satz 4.2 from \cite{Ne73}.
%: {\bf KB: I think we should rather give a reference than a proof.}
\begin{lemma}\label{fourier expansion}
	For $z$, $\tau\in \mathfrak F$  with $z_2> \max\{\im(\gamma \tau): \gamma\in \Gamma\}$, we have, for $\re(s)>1$,
	\begin{equation}\label{Fourier exp Green}
		G_s(z,\tau) =\frac{z_2^{1-s}}{1-2s} E(\tau;s)-\sqrt{z_2}\sum_{n\in\Z\setminus \{0\}} F_{-n,s}(\tau) K_{s-\frac12}(2\pi |n| z_2) e^{2\pi inz_1}.
	\end{equation}
If  $K\geq 1$, then for $\ell\in\{1,\ldots,K\}$, $\tau \in \H$, $z\in \mathfrak F$ with $\im(\sigma_\ell^{-1} z)>\im(\sigma_\ell^{-1} \gamma \tau)$, the Fourier expansion at the cusp $\varrho_\ell$ is given by
\begin{multline}\label{eq. Four exp Green at cusps}
G_s(z,\tau)=\frac{\im\left(\sigma_\ell^{-1} z\right)^{1-s}}{1-2s}E_{\varrho_\ell}(\tau;s)  %\hspace{2cm}\text{{\bf KB: Isn't there a $n$ missing? $\downarrow$}}\\
\\- \sqrt{\im\left(\sigma_\ell^{-1} z\right)} \sum_{n\in \Z\setminus\{0\}}F_{-n,s,\varrho_\ell}(\tau)K_{s-\frac12}\left(2\pi |n|\im\left(\sigma_\ell^{-1} z\right)\right)e^{2\pi i n\re\left(\sigma_\ell^{-1} z\right)}.
\end{multline}
\end{lemma}
An immediate consequence of Lemma \ref{fourier expansion} is the following corollary.

\begin{corollary}\label{cor. green's bounds other cusps}
  Assume that $K\geq 1$.  Let $Y$ be sufficiently large such that for $z\in \sigma_\ell P(Y)$ and $\ell\in\{1,\ldots,K\}$ we have
  that $Y:=\im(\sigma_\ell^{-1} z)>\im(\sigma_\ell^{-1} \gamma \tau)$ for all $\gamma\in \Gamma$.
  \begin{enumerate}[leftmargin=*,label=\rm(\arabic*)]
  \item As $Y\to\infty$, we have
 \begin{equation*}\label{eq. deriv G+E bound}
\left[\frac{\partial}{\partial s}\left(G_s(\sigma_\ell  z,\tau) \!+\! E(\tau;s)\right)\right]_{s=1}\!\!\!= O\!\left(P_{1}(\tau)+P_{\ell,0,1}(\tau)  \!+\! P_{\varrho_\ell,i\infty}(\tau)\log (Y) \!+\!\log^2(Y)\right)\!,
  \end{equation*}
  where the implied constant is independent of $\tau\in\mathfrak F$.
  \item As $Y\to\infty$, we have
\begin{equation*}\label{eq. norm deriv G+E bound}
\left[\frac{\partial}{\partial s} \frac{\partial}{\partial\mathrm n}\left(G_s(\sigma_\ell z,\tau) + E(\tau;s)\right)\right]_{s=1}= O\left(\frac{\log (Y)+P_{\varrho_\ell,i\infty}(\tau)}{Y}\right)\!,
  \end{equation*}
   where the implied constant is independent of $\tau\in\mathfrak F$.
  \end{enumerate}
   \end{corollary}

   \begin{proof}
    \begin{enumerate}[wide,labelwidth=0pt,labelindent=0pt]
     \item The assumption $Y=\im(\sigma_\ell^{-1} z)>\im(\sigma_\ell^{-1} \gamma \tau)$, together with the bound \eqref{eq. NP growth at other cusps} and the exponential decay of the derivative of the $K$-Bessel function imply that the function on the right-hand side of \eqref{eq. Four exp Green at cusps} is uniformly bounded in $\tau$ as $Y\to \infty$. Furthermore, from \eqref{KronLimitPArGen} and \eqref{eq. KronLimitPAr cusps},
    	\begin{equation*}
    	\noindent \left[\frac{\partial}{\partial s}\left( E(\tau;s) \!-\!E_{\varrho_\ell}(\tau;s)\frac{Y^{1-s}}{2s-1}\right)\right]_{s=1}\!=\!O\left(P_{1}(\tau)\!+\!P_{\ell,0,1}(\tau)\!+\! P_{\varrho_\ell,i\infty}(\tau) \log \left(Y\right)\! +\!\log^2 \left(Y\right)\right),
    	\end{equation*}
    	as $Y\to \infty$, where the implied constant is independent of $\tau$. This proves (1).
    	
    	\item This is shown analogously to (1), so we omit the proof. \qedhere
    \end{enumerate}
   \end{proof}

\subsection{Relating $[\frac{\partial}{\partial s}\xi_{0,z}(G_{\bar s}(z,\tau))]_{s=1}$ to $\mathbb{F}(z,\tau)$}\label{Subsection}

Using \eqref{Fourier exp Green}, we prove the following proposition.
\begin{proposition}\label{prop. holom part}
  The function $z\mapsto\mathbb{F}(z,\tau)$ is the holomorphic part of the weight two biharmonic Maass form
  %{\bf KB: Do we need $\overline s$? I think we also did that at some point but do not remember why.}
  $[\frac{\partial}{\partial s}\xi_{0,z}(G_{\bar s}(z,\tau))]_{s=1}.$
\end{proposition}

\begin{proof}

	Note that $f_{s,\tau}(z):=\xi_{0,z}(G_{\bar s}(z,\tau))$ is holomorphic at $s=1$ and $G_s(z,\tau)$ is an eigenfunction of $\Delta_{0,z}=-\xi_{2,z}\circ\xi_{0,z}$ with eigenvalue $s(1-s)$. Therefore, $f_\tau(z):=f_{1,\tau}(z)$ is biharmonic.

  Acting by $\xi_{0,z}$ on the Fourier expansion \eqref{Fourier exp Green}, for $\sigma>1$, we have, using that $\overline{E(\tau;\overline s)} = E(\tau;s)$, $\overline{F_{-n,\overline s}(\tau)} = F_{n,s}(\tau)$, and $K_{\nu-1}(x)+K_{\nu+1}(x)=-2K_\nu'(x)$ (by 8.486.11 of \cite{GR07})
  \begin{align*}
    f_{s,\tau}(z)&=\frac{(s-1)z_2^{-s}}{2s-1}E(\tau;s) -\frac{1}{2\sqrt{z_2}} \sum_{n\in\Z\setminus \{0\}} K_{s-\frac12}(2\pi |n| z_2) F_{-n,s}(\tau)e^{{2\pi inz_1}}\\
    &+\pi\sqrt{z_2} \sum_{n\in\Z\setminus \{0\}}\left( 2nK_{s-\frac12}(2\pi |n| z_2)+|n|\left( K_{s-\frac32}(2\pi |n| z_2)+K_{s+\frac12}(2\pi |n| z_2)\right)\right) F_{-n,s}(\tau)e^{2\pi inz_1}\\&=: T_s^{[1]}(z,\tau)+T_s^{[2]}(z,\tau)+T_s^{[3]}(z,\tau).
  \end{align*}
 The right-hand side is holomorphic at $s=1$ and equals the sum of the three terms at $s=1$. Using \eqref{KronLimitPArGen}, we see that the derivative at $s=1$ of the first term equals
  $$
 \left[\frac{\partial}{\partial s}T_s^{[1]}(z,\tau)\right]_{s=1} = \frac{\beta\vol(\Gamma\backslash\mathbb{H}) - P(\tau)-2-\log (z_2)}{\vol(\Gamma\backslash\mathbb{H}) z_2}.
  $$
  By \cite[8.486(1).21]{GR07}, we have
   $$
   \left[\frac{\partial}{\partial s}K_{s-\frac{1}{2}}(x)\right]_{s=1}= - \sqrt{\frac{\pi}{2x}}e^x\mathrm{Ei}(-2x),
   $$
    where  for $x>0$, $\mathrm{Ei}(-x) := \int\limits_{-\infty}^x e^t \frac{dt}{t} $ is the {\it exponential integral}.  With this, we get that
 $$
  \left[\frac{\partial}{\partial s}T_s^{[2]}(z,\tau)\right]_{s=1} =\sum_{n\in\Z\setminus \{0\}}\left( \mathrm{Ei}(-4\pi|n|z_2)e^{4\pi|n|z_2}F_{-n,1}(\tau) - \left[\frac{\partial}{\partial s}F_{-n,s}(\tau)\right]_{s=1}\right)\frac{e^{-2\pi|n|z_2+2\pi i nz_1}}{4z_2\sqrt{|n|}}.
 $$
This does not contribute to the holomorphic part of $f_\tau$.

The derivative at $s=1$ of $T_s^{[3]}(z,\tau)$ equals the sum of two terms, namely,
  \begin{multline}\label{eq. T3 term 1}
\pi\sqrt{z_2}\!\!\! \sum_{n\in\Z\setminus \{0\}} \!\! \left[\!\frac{\partial}{\partial s}\!\left( 2nK_{s-\frac12}(2\pi |n| z_2)\!+\!|n|\!\left(\! K_{s-\frac32}(2\pi |n| z_2)\!+\!K_{s+\frac12}(2\pi |n| z_2)\!\right)\!\right)\!\right]_{s=1}\!\!\!\!
 F_{-n,1}(\tau)e^{2\pi inz_1},\!\!\!
\end{multline}
  \begin{equation}\label{eq. T3 term 2}
 	\pi\sqrt{z_2}\!\!\!\sum\limits_{n\in\Z\setminus\{0\} }\!\!\left(2|n|K_{\frac12}(2\pi |n| z_2)+ n\left(K_{-\frac12}(2\pi |n| z_2)+K_{\frac32}(2\pi |n| z_2)\right)\right)\left[\frac{\partial}{\partial s}F_{-n,s}(\tau)\right]_{s=1}\!\!e^{2\pi inz_1}.
 \end{equation}
Reasoning analogously as above, using 8.486.11 and 8.486(1) of \cite{GR07}, we conclude that \eqref{eq. T3 term 1}  does not contribute to the holomorphic part of $f_\tau$ either.
Therefore, the holomorphic part of $f_\tau$ equals the holomorphic part of \eqref{eq. T3 term 2}.
% {\bf KB: Why isn't there a contribution from differentiating the $K$-Bessel function? LS: I omitted that contribution below because it does not contribute to the holomorphic part; the holomorphic part appears only in the $K$-Bessels which multiply $\frac{\partial}{\partial s}F_{-n,s}(\tau)$. KB: I think we should say more here.}
Using \eqref{eq:Bessel_identities1} and
$
K_{-\frac{1}{2}}(x)=K_{\frac32}(x)=\sqrt{\frac{\pi}{2x}}(1+\frac{1}{x})e^{-x},
$
 we deduce that the holomorphic part of $f_\tau$ equals $\mathbb{F}(\tau,z)$
 which completes the proof.\qedhere
\end{proof}

\section{Proof of Theorem \ref{thm: main 1} and Corollary \ref{cor. divisor}}\label{sec:Proof_of_Thm1}

In this section we prove Theorem \ref{thm: main 1} and Corollary \ref{cor. divisor}.
\begin{proof}[Proof of Theorem \ref{thm: main 1}]
	\begin{enumerate}[wide,labelwidth=0pt,labelindent=0pt]
		\item Assume first that $z_2>\im(\gamma \tau)$, for all $\gamma\in\Gamma$. Using \eqref{Fourier exp Green}, we obtain
		\begin{equation*}%\label{eq:G_expansion}
		G_s(z,\tau) + E(\tau;s)\\
		=\left( 1+\frac{z_2^{1-s}}{1-2s}\right) E(\tau;s)-\sqrt{z_2}\sum_{n\in\Z\setminus \{0\}} F_{-n,s}(\tau)e^{2\pi inz_1}K_{s-\frac12}(2\pi |n| z_2) .
		\end{equation*}
		The second term on the right-hand side of \eqref{Fourier exp Green} is holomorphic at $s=1$. Moreover, using \eqref{KronLimitPArGen},
		\begin{equation}
		\lim_{s\to 1} \left( 1+\frac{z_2^{1-s}}{1-2s}\right)E(\tau;s) = \frac{2+\log(z_2)}{\operatorname{vol}(\Gamma\backslash \mathbb{H})}.\label{eq. lim of Eis}
		\end{equation}
Using this, \eqref{G}, \eqref{eq:Bessel_identities1}, and \eqref{eq. lim of Eis}, we get
		\begin{equation*}%\label{eq. lim as s to 1 of g+e}
		\mathcal{G}(z,\tau) \!=\! \frac{2+\log(z_2)}{\operatorname{vol}(\Gamma\backslash \mathbb{H})} -\frac12\sum_{n\in\Z\setminus \{0\}}\frac{F_{-n,1}(\tau)}{\sqrt{|n|}}e^{2\pi i (nz_1+i|n|z_2)}.
		\end{equation*}
		Recalling that $R_{0,z}=2i\frac{\partial}{\partial z}$, we obtain
		\begin{equation}\label{eq. lim as s to 1}
		R_{0,z}\left(\mathcal{G}(z,\tau)\right) = \frac{1}{\operatorname{vol}(\Gamma\backslash \mathbb{H})z_2} + 2\pi\sum_{n\geq 1}\sqrt{n}F_{-n,1}(\tau)e^{2\pi i n z}.
		\end{equation}
		The left-hand side of \eqref{eq. lim as s to 1} is a real-analytic function for all $z,\tau\in\mathbb{H}$ with $z\neq \gamma \tau$ for all $\gamma\in\Gamma$.  Therefore, by analytic continuation, \eqref{eq. lim as s to 1} holds true for all these $z,\tau\in\mathbb{H}$. In view of the defining equation \eqref{eq. defn gen funct F} for $\mathcal F(z,\tau)= \mathcal F_1(z,\tau)$, (1) holds.
		
		\item By part (1), we only need to show that $\tau\mapsto \mathcal{G}(z,\tau)$ is a weight zero polar harmonic Maass form. From the properties of $G_s(z,\tau)$ and $E(\tau;s)$, as recalled in Subsections \ref{sec. parab Eis} and \ref{sec. aut green},
it is immediate that $\mathcal{G}(z,\tau)=\lim_{s\to 1}\left(G_s(z,\tau) + E(\tau;s)\right)$ is $\Gamma$-invariant and real-analytic in both variables.  Also, it is annihilated by the Laplacian $\Delta_{0,\tau}$ and it has a logarithmic singularity in the fundamental domain as $\tau\to z$. Thus, $\tau\mapsto \mathcal{G}(z,\tau)$ is a weight zero polar harmonic Maass form, and, in view of \eqref{eq. lim as s to 1}, so is function $\tau\mapsto \mathcal{F}(z,\tau)$.

\item Recall that $\widehat{\mathcal{F}}(z,\tau) = R_{0,z}(\mathcal{G}(z,\tau))$. From the properties of the Green's function and the parabolic Eisenstein series it follows that $z\mapsto \mathcal{G}(z,\tau)$ is $\Gamma$-invariant, real-analytic in $z$, and annihilated by the Laplacian $\Delta_{0,z}$. Thus, $z\mapsto \widehat{\mathcal{F}}(z,\tau)$ is a weight two polar harmonic Maass form.
We are left to determine the polar part of $\widehat{\mathcal F}(z,\tau)$.
Recall the definition of $k_s$ in \eqref{ks}. The identity
$$
\left(1-\frac{|z-\tau|^2}{|z-\overline\tau|^2}\right)^{-1}=\frac{|z-\overline\tau|^2}{4z_2\tau_2}
=1+\frac{|z-\tau|^2}{4z_2\tau_2}=1+u(z,\tau),
$$
enables us to write
$$
k_s(z,\tau)= -\frac{\Gamma^2(s)}{4\pi\Gamma(2s)} \left(1+u(z,\tau)\right)^{-s} {}_2F_1\left(s,s;2s;\frac{1}{1+u(z,\tau)}\right).
$$
Using 15.5.3 of \cite{DLMF}, we obtain
$$
\frac{\partial}{\partial z}k_s(z,\tau)=\frac{s\Gamma^2(s)}{4\pi \Gamma(2s)} \left(1+u(z,\tau)\right)^{-s-1}{}_2F_1\left(s+1,s;2s;\frac{1}{1+u(z,\tau)}\right)\frac{\partial}{\partial z}u(z,\tau).
$$
Next note that
\begin{equation*}%\label{eq:partial_u_formula}
\frac{\partial}{\partial z}u(z,\tau)= \frac{i(\overline z - \overline \tau)(\overline z - \tau)}{8z_2^2\tau_2}.
\end{equation*}
Combining this with the final equation in 9.131.1 of \cite{GR07}, we obtain
$$
\frac{\partial}{\partial z}k_s(z,\tau)=\frac{\Gamma(s)\Gamma(s+1)}{4\pi \Gamma(2s)}\left(1+u(z,\tau)\right)^{-s}\frac{1}{u(z,\tau)}{}_2F_1\left(s-1,s;2s;\frac{1}{1+u(z,\tau)}\right)\frac{i(\overline z - \overline \tau)(\overline z - \tau)}{8z_2^2\tau_2}.
$$
Using the definition of $u(z,\tau)$ and some elementary algebraic manipulations, we get
$$
\frac{1}{4u(z,\tau)(1+u(z,\tau))}\frac{i(\overline z - \overline \tau)(\overline z- \tau)}{8z_2^2\tau_2}= \frac{i\tau_2}{2(z-\tau) (z-\overline \tau)}.
$$
Hence, we obtain
\begin{equation}\label{eq}
\frac{\partial}{\partial z}k_s(z,\tau)= \frac{i\tau_2\Gamma(s)\Gamma(s+1)}{2\pi(z-\overline \tau)(z-\tau)\Gamma(2s)}\left(1+u(z,\tau)\right)^{1-s}
{}_2F_1\left(s-1,s;2s;\frac{1}{1+u(z,\tau)}\right).
\end{equation}
We now apply 9.122 of \cite{GR07} with $\alpha=s-1$, $\beta=s$, and $\gamma=2s$. Thus, for $s\in\C$ with $\re(s)\geq 1$, as $z\to\tau$
(so $u(z,\tau)\to 0$), we have that
\begin{equation*}
\frac{\Gamma(s)\Gamma(s+1)}{ \Gamma(2s)}\left(1+u(z,\tau)\right)^{1-s}{}_2F_1\left(s-1,s;2s;\frac{1}{1+u(z,\tau)}\right)
\sim 1.
\end{equation*}
Plugging this into \eqref{eq} proves that, for $s\in\C$ with $\re(s)>0$,
$$
\frac{\partial}{\partial z}k_s(z,\tau)\sim\frac{i\tau_2}{2\pi(z-\overline \tau)(z-\tau)}\,\,\,\text{\rm as}\,\,\, z\to\tau.
$$
By \eqref{eq. defn Green}, $\frac{\partial}{\partial z} k_s(\gamma z,\tau)$ has, for $s\in\C$ with $\re(s)>1$ and $\gamma\in\Gamma$, a singularity if $\tau=\gamma z$. Then
\begin{equation}\label{eq2}
\frac{\partial}{\partial z}(G_s(z,\tau) + E(\tau;s))=\frac{\partial}{\partial z}G_s(z,\tau)\sim \frac{i \mathrm{Stab}_\tau \tau_2}{2\pi(z-\overline \tau)(z-\tau)}\quad (\text{as}\,\,z\rightarrow \tau).
\end{equation}
Since the right-hand side is independent of $s$ and the left-hand side is well-defined for $s=1$ (and by \eqref{G} at $s=1$ equals $\frac{\partial}{\partial z}\mathcal{G}(z,\tau)$), by multiplying \eqref{eq2} with $2i$, we deduce that the polar part of
$
2i\frac{\partial}{\partial z}\mathcal{G}(z,\tau) - \frac{1}{\operatorname{vol}(\Gamma\backslash \mathbb{H}) z_2},
$
as $z\to \tau$, is $-\frac{\mathrm{Stab}_\tau \tau_2}{\pi(z-\tau)(z-\overline \tau)}$.\qedhere
%The proof follows by writing
%\begin{equation*}
%	-\frac{\mathrm{Stab}_\tau\cdot \tau_2}{\pi(z-\tau)(z-\overline \tau)}
%	= \frac{ \mathrm{Stab}_\tau}{2\pi i} \left(\frac{1}{z-\overline{\tau}}-\frac{1}{z-\tau}\right). \qedhere
%\end{equation*}
\end{enumerate}
\end{proof}

We are now ready to prove Corollary \ref{cor. divisor}.

\begin{proof}[Proof of Corollary \ref{cor. divisor}]  We let
	\begin{equation*}
		g(z):=f^{{\rm div}} (z)-\frac{k}{4\pi z_2}-\frac{\Theta\left(f(z)\right)}{f(z)}.
	\end{equation*}
	To prove the Corollary, we need to show that $g\in S_2(\Gamma)$. By the valence formula, we have that
$$
	\frac{1}{\vol(\Gamma\backslash\mathbb{H}) } \sum_{\tau\in\mathfrak F} \frac{\ord_\tau(f)}{{\rm Stab}_\tau}=\frac{k}{4\pi }.
$$
Combining \eqref{eq. div dom form} with the definition of $\widehat{\mathcal F}$, we get
\begin{equation}\label{eq. f div difference}
g(z)=\sum_{\tau\in\mathfrak F} \frac{\ord_\tau(f)}{{\rm Stab}_\tau}\mathcal{F}(z,\tau) + \frac{\Theta(f(z))}{f(z)}.
\end{equation}
Using the definition of the weight $k$ rasing operator $R_{k,z}:= 2i\frac{\partial}{\partial z} + \frac{k}{z_2}$, it is not hard to see that $$\frac{\Theta(f(z))}{f(z)}-\frac{k}{4\pi z_2}= -\frac{1}{4\pi}\frac{R_{k,z}(f(z))}{f(z)}.$$
This gives that
\begin{equation}\label{gR}
	g(z)=\sum_{\tau\in\mathfrak{F}} \frac{\ord_\tau (f)}{{\rm Stab}_\tau} \widehat{\mathcal{F}}(z,\tau)-\frac{1}{4\pi} \frac{R_{k,z}\left(f(z)\right)}{f(z)}.
\end{equation}
Thus, by Theorem \ref{thm: main 1} (3), $g$ is harmonic of weight two.

We next show that $g$ does not have poles in $\H$. For this, we use the representation of \eqref{eq. f div difference}.
The residue of $\frac{\Theta(f(z))}{f(z)}$ at $z=\tau\in\mathbb{H}$ is $\frac{\ord_\tau(f)}{2\pi i} $. By Theorem \ref{thm: main 1} (3), the residue of $\mathcal{F}(z,\tau)$ at a simple pole $z=\tau$ is $-\frac{{\rm Stab}_\tau}{2\pi i} $. Thus, $g$ has no poles in $\mathbb{H}$, hence $g$ is a harmonic Maass form of weight two for $\Gamma$.
We are left to prove that $g$ vanishes at the cusps of $\Gamma$. Thus, we need to show that, for a scaling matrix $\sigma \in\SL_2(\R)$ associated to a cusp $\varrho =i\infty$ or $\varrho=\varrho_\ell$ for $\ell\in\{1,\ldots,K\}$, we have
\begin{equation*}
	\lim\limits_{z\to i\infty} g(z)|_2 \sigma =0.
\end{equation*}
From \eqref{gR}, we have
\begin{equation}\label{eq. g slashed}
	g|_2 \sigma(z)=\sum_{\tau\in\mathfrak{F}} \frac{\ord_\tau(f)}{{\rm Stab}_\tau} \widehat{\mathcal{F}}(z,\tau)|_{2,z} \sigma -\frac{1}{4\pi} \frac{R_{k,z}\left(f(z)|_k \sigma\right)}{f(z)|_k \sigma}.
\end{equation}
By Theorem \ref{thm: main 1} (3), $z\mapsto \widehat{\mathcal{F}}(z,\tau)$ is a polar harmonic Maass form of weight zero with a simpe pole only as $z\to\tau$. Moreover, from \eqref{eq. lim as s to 1} it is evident that $\lim_{z\to i\infty}\widehat{\mathcal{F}}(z,\tau)=0$. If $\varrho=\varrho_\ell\neq i\infty$ and $\tau\neq \varrho$, we employ the Fourier expansion \eqref{eq. Four exp Green at cusps} of the Green's function at the cusp $\varrho$, combined with the Laurent series expansions \eqref{KronLimitPArGen} and \eqref{eq. KronLimitPAr cusps} of the Eisenstein series at cusps $i\infty$ and $\varrho=\varrho_\ell$ to deduce, for  $z_2$ sufficiently large, that
\begin{multline*}
\widehat{\mathcal{F}}(\sigma z,\tau)=R_{0,z}\bigg(\frac{2+\log (z_2)}{\vol(\Gamma\backslash \mathbb{H})} +\beta-\frac{P(\tau)}{\vol(\Gamma\backslash \mathbb{H})}-\beta_{\ell,0} + \frac{P_{\varrho,i\infty}(\tau)}{\vol(\Gamma\backslash \mathbb{H})}
 - \frac{1}{2}\sum_{n\in\Z\setminus\{0\}}\frac{ F_{-n,1,\varrho}(\tau)}{\sqrt{|n|}}e^{2\pi i (nz_1+i|n|z_2)}\bigg).
\end{multline*}
Therefore, we have
$$
\widehat{\mathcal{F}}(\sigma z,\tau)=\frac{1}{\vol(\Gamma\backslash \mathbb{H})z_2}+2\pi \sum_{n\geq 1}\sqrt{n}F_{-n,1,\varrho}(\tau)e^{2\pi i nz}.
$$
If $\tau=\varrho$, then $\ord_\tau(f)=0$. This proves that the first term on the right-hand side of \eqref{eq. g slashed} vanishes as $z\to i\infty$, for all $\tau$.
For the second term, we write
\begin{equation*}
	f(z)|_k \sigma=\sum_{n\ge 0} c_{f,\sigma}(n)e^{2\pi i n z}
\end{equation*}
with $c_{f,\sigma}(0)\neq 0$. Then
\begin{equation*}
	f(z)|_k \sigma=c_{f,\sigma}(0)+O\left(e^{-2\pi  z_2} \right),\quad R_{k,z}\left(f(z)|_k \sigma\right)=\frac{kc_{f,\sigma}(0)}{z_2}+O\left(e^{-2\pi z_2}\right).
\end{equation*}
This directly implies the claim.
\end{proof}

\section{Proof of Theorem \ref{thm: main 2_1} and Theorem \ref{thm: main 2_2}}\label{sec:Proof_of_series}

Using \eqref{eq. defn gen funct F} and \eqref{eq. N-P series defn}, we write
\begin{align}\label{eq:F_double_series}
\mathcal F_{s}(z,\tau)=2\pi\sum_{n\ge1}\sqrt{n}\sum_{\gamma\in\Gamma_\infty\setminus\Gamma} \sqrt{\im(\gamma\tau)}
I_{s-\frac12}\left(2\pi n\im(\gamma \tau)\right)e^{-2\pi in\re(\gamma\tau)}e^{2\pi inz}.
\end{align}

We next prove Theorem \ref{thm: main 2_1}.

\begin{proof}[Proof of Theorem \ref{thm: main 2_1}]
Noting that $I_{\frac{3}{2}}(x) = \sqrt{\frac{2}{\pi x}} (\cosh(x) - \frac{\sinh(x)}{x})$, we have
$$
\sqrt{n} \sqrt{\im(\gamma\tau)}
I_{\frac32}\left(2\pi n\im(\gamma \tau)\right)
= \frac{1}{\pi}\left(\cosh(2\pi n\im(\gamma \tau)) - \frac{\sinh(2\pi n\im(\gamma \tau))}{2\pi n\im(\gamma \tau)}\right).
$$
Thus \eqref{eq:F_double_series} implies that
$$
\mathcal F_{2}(z,\tau)=2\sum_{n\ge1}\sum_{\gamma\in\Gamma_\infty\setminus\Gamma}
\left(\cosh(2\pi n\im(\gamma \tau)) - \frac{\sinh(2\pi n\im(\gamma \tau))}{2\pi n\im(\gamma \tau)}\right)
e^{2\pi in(z-\re(\gamma\tau))}.
$$
Assume that $z,\tau\in\mathbb{H}$ are fixed with $z_2>\im(\gamma \tau)$ for all $\gamma\in\Gamma$.
For any fixed $\gamma$, we can write
\begin{multline}\label{eq:geometric_series}
2\sum_{n\ge1} \cosh(2\pi n\im(\gamma \tau)) e^{2\pi in(z-\re(\gamma\tau))}
= \sum_{n\ge1} \left(e^{2\pi n\im(\gamma \tau)} + e^{-2\pi n\im(\gamma \tau)}\right) e^{2\pi in(z-\re(\gamma\tau))}
\\= \sum_{n\ge1} \left(e^{2\pi i n(z-\gamma\tau)} + e^{2\pi i n(z-\overline{\gamma\tau})}\right)
= \frac{e^{2\pi i (z-\gamma\tau)}}{1-e^{2\pi i (z-\gamma\tau)}} +
 \frac{e^{2\pi i (z-\overline{\gamma\tau})}}{1-e^{2\pi i (z-\overline{\gamma\tau})}}.
\end{multline}
Similarly, we compute
\begin{multline*}
\hspace{-.3cm}2\sum_{n\ge1} \frac{\sinh(2\pi n\im(\gamma \tau))}{2\pi n\im(\gamma \tau)} e^{2\pi in(z-\re(\gamma\tau))}
= \frac{1}{2\pi \im(\gamma \tau)}\sum_{n\ge1} \frac{1}{n} \! \left(e^{2\pi n\im(\gamma \tau)} - e^{-2\pi n\im(\gamma \tau)}\right) \! e^{2\pi in(z-\re(\gamma\tau))}
\\\hspace{.6cm}= \frac{1}{2\pi \im(\gamma \tau)}\sum_{n\ge1} \frac{1}{n}\! \left(e^{2\pi i n(z-\gamma\tau)} \!-\! e^{2\pi i n(z-\overline{\gamma\tau})}\right)
= \frac{\Log\!\left(1-e^{2\pi i\left(z-\overline{\gamma\tau}\right)}\right)
-\Log\!\left(1-e^{2\pi i\left(z-\gamma\tau\right)}\right)}
{\pi i( \overline{\gamma\tau}-\gamma \tau)}.\qedhere
\end{multline*}
\end{proof}

\begin{remark}  For $k \in \mathbb{N}_{>1}$, one can similarly compute
of $\mathcal F_{k}(z,\tau)$ using known evaluations of $I_{\frac{k+1}{2}}$.
\end{remark}
To prove Theorem \ref{thm: main 2_2}, we require the following lemma.

\begin{lemma}\label{lem:Bessel_bound} There exists a constant $C$ and a sufficiently small neighborhood of $s$ near one such that, for all $x \geq 0$, we have
\begin{equation}\label{eq:I-sinh-bound}
	\left| I_{s-\frac{1}{2}}(x) - \frac{\sinh(x)x^{s-\frac{3}{2}}}{2^{s-\frac{1}{2}}\Gamma\left(s+\frac{1}{2}\right)}\right| \leq C(s-1)x^{2}e^{x}.
\end{equation}
\end{lemma}

\begin{proof}
	We have to show that
	\begin{equation*}
		f_s(x):=\frac{e^{-x}}{(s-1)x^2} \left(I_{s-\frac{1}{2}}(x)-\frac{\sinh(x)x^{s-\frac{3}{2}}}{2^{s-\frac{1}{2}}\Gamma\left(s+\frac{1}{2}\right)}\right)
	\end{equation*}
	is uniformly bounded in $s$ and $x$ for $s$ close to one. Using that $I_\frac{1}{2}(x)=\sqrt{\frac{2}{\pi x}}$, it is not hard to see that $f_s(x)$ does not have a pole in $s=1$. Thus, we may consider $s$ as fixed. Now the claim follows since $\lim\limits_{x\to 0}f_s(x)$ and $\lim\limits_{x\to \infty}f_s(x)$ exist. For this we use that
	\begin{align*}
		&I_{s-\frac{1}{2}}(x)=\left(\frac{x}{2}\right)^{s-\frac{1}{2}}\left(\frac{1}{\Gamma\left(s+\frac{1}{2}\right)}+O\left(x^2\right)\right), \quad
		\sinh(x)=x+O\left(x^3\right)\qquad \text{as}~ x\to 0,\\
		&I_{s-\frac{1}{2}}(x)\sim \frac{e^x}{\sqrt2\pi x}, \quad \sinh(x)\sim \frac{e^x}{2} \hspace{6.4cm} \text{as}~x\to \infty.\qedhere
	\end{align*}

%Finally, the tail of the series expansions (1.12) and (1.15) from \cite{Ne17} is estimated from which \eqref{eq:IBessel_asymp}
%follows; see Theorem 1.6, and more specifically equation (1.30) and subsequent discussion, from \cite{Ne17}.

%By combining \eqref{eq:sinh_asymp} with \eqref{eq:IBessel_asymp}, one gets the \eqref{eq:I-sinh-bound} for $x \geq x_{0}$
%and thus completes the proof of Lemma \ref{lem:Bessel_bound}.
\end{proof}

\begin{remark} As the proof of Lemma \ref{lem:Bessel_bound} shows, asymptotic formulas for \eqref{eq:I-sinh-bound} exist near zero and
$i\infty$.  For our purposes, the bound given in \eqref{eq:I-sinh-bound} suffices.
\end{remark}

\begin{proof}[Proof of Theorem \ref{thm: main 2_2}]
Note that \eqref{eq. defn gen funct F} does not converge absolutely at $s=1$.
So we consider
\begin{equation}\label{eq:double_series}
\mathcal F_{s,w}(z,\tau):= 2\pi\sum_{n\ge1}n^{\frac{w+1}{2}}F_{-n,s}(\tau)e^{2\pi inz},
\end{equation}
for $w\in\mathbb{C}$ close to zero. The bounds
\eqref{eq. bound for NP at other cusps}
for $F_{-n,s}$ imply that \eqref{eq:double_series} is
holomorphic in $w$ and $s$ close to $w=0$ and $s=1$
if $\im(z) > \im(\gamma\tau)$ for all $\gamma$.  As such, the discussion in Subsection 0.2
of \cite{Krantz} implies that $\mathcal F_{s,w}(z,\tau)$ can be written as a convergent
power series in $w$ and $s$ near $w=0$ and $s=1$.  In particular,
all directional limits toward $(1,0)$ are equal, thus
\begin{align*}
	\lim\limits_{(s,w)\rightarrow (1,0)}
	\mathcal F_{s,w}(z,\tau)&=
	\lim\limits_{s\rightarrow 1}
	\mathcal F_{s,0}(z,\tau)=
	\lim\limits_{s\rightarrow 1} \mathcal F_{s}(z,\tau)=
	\lim\limits_{s\rightarrow 1} \mathcal F_{s,1-s}(z,\tau).
\end{align*}
In other words, we have that
$$
F_{s,1-s}(z,\tau)= \sum\limits_{n\ge0} a_{n}(s-1)^{n}
$$
in a sufficiently small neighborhood of $s=1$. In particular, we have
$$
a_{0} = \lim\limits_{s\to 1^+} F_{s,1-s}(z,\tau) = \mathcal F_{1}(z,\tau).
$$
Combining this with Lemma \ref{lem:Bessel_bound}, calculations such as those leading to \eqref{eq:geometric_series} give
$$
\mathcal F_{s,1-s}(z,\tau) = \frac{2^{1-s}\sqrt{\pi}}{\Gamma\!\left(s-\frac{1}{2}\right)}
\sum_{\gamma\in\Gamma_\infty\setminus\Gamma}
\im(\gamma \tau)^{s-1}\left(\frac{1}{1-e^{2\pi i (z-\gamma\tau)}} - \frac{1}{1-e^{2\pi i (z-\overline{\gamma\tau})}}\right)
+O(s-1)
\,\,\,\,\,
\text{\rm as $s \rightarrow 1$},
$$
where the O-term depends on $z$ and $\tau$, and is uniform if $z$ and $\tau$ lie in compact sets.
Note
$$
\frac{2^{1-s}\sqrt{\pi}}{\Gamma\left(s-\frac{1}{2}\right)} =  1 + O(s-1)
\,\,\,\,\,
\text{\rm as $s \rightarrow 1$.}
$$
Because $\mathcal F_{s,1-s}(z,\tau)$ is holomorphic at $s=1$, we get that
\begin{equation}\label{eq:asymp_near_zero}
\mathcal F_{s,1-s}(z,\tau) =
\sum_{\gamma\in\Gamma_\infty\setminus\Gamma}
\im(\gamma \tau)^{s-1}\left(\frac{1}{1-e^{2\pi i (z-\gamma\tau)}} - \frac{1}{1-e^{2\pi i (z-\overline{\gamma\tau})}}\right)
+O(s-1)
\,\,\,\,\,
\text{\rm as $s \rightarrow 1$}
\end{equation}
For simplicity, we make the change of variables $w:=s-1$. From the above,
\begin{equation}\label{eq:weighted_Eisen}
\sum_{\gamma\in\Gamma_\infty\setminus\Gamma}
\im(\gamma \tau)^{w}\left(\frac{1}{1-e^{2\pi i (z-\gamma\tau)}} - \frac{1}{1-e^{2\pi i (z-\overline{\gamma\tau})}}\right)
\end{equation}
admits a holomorphic continuation from $\re(w) > 0$ to a neighborhood which contains $w = 0$. Furthermore, we have
$$
\mathcal F_{1}(z,\tau) = \lim\limits_{w \rightarrow 0^+}\sum_{\gamma\in\Gamma_\infty\setminus\Gamma}
\im(\gamma \tau)^{w}\left(\frac{1}{1-e^{2\pi i (z-\gamma\tau)}} - \frac{1}{1-e^{2\pi i (z-\overline{\gamma\tau})}}\right).
$$
Thus, we can rewrite \eqref{eq:asymp_near_zero} as
\begin{equation}\label{eq:asymp_formula}
\sum_{\gamma\in\Gamma_\infty\setminus\Gamma}
\im(\gamma \tau)^{w}\left(\frac{1}{1-e^{2\pi i (z-\gamma\tau)}} - \frac{1}{1-e^{2\pi i (z-\overline{\gamma\tau})}}\right)
= \mathcal F_{1}(z,\tau) + O(w)
\,\,\,\,\,
\text{\rm as $w \rightarrow 0^{+}$}.
\end{equation}

We next apply Theorem 11.1 of \cite{Korevaar} in conjunction with \eqref{eq:counting2_Eisen}.
To translate \eqref{eq:asymp_formula} to the notation of (11.2)
of \cite{Korevaar}, one sets $u=\xi^{-1}$, $\{\lambda_{n}\} = \{-\log(\im(\gamma\tau))\}$, and
$$
\{a_{n}\} = \left\{\frac{1}{1-e^{2\pi i (z-\gamma\tau)}} - \frac{1}{1-e^{2\pi i (z-\overline{\gamma\tau})}}
\right\}.
$$
Note that
$$
\frac{1}{1-e^{2\pi i (z-\gamma\tau)}} - \frac{1}{1-e^{2\pi i (z-\overline{\gamma\tau})}} =
 \frac{-\pi \im(\gamma\tau)}{\sin^{2}(2\pi(z-\re(\gamma\tau)))} + O_{z}\left(\im(\gamma\tau)^{2}\right)
$$
with a universal constant universally bounded from below, which follows because of
$z_2 > \im(\gamma\tau)$ holds for all $\gamma$.
As a result, (11.1) of \cite{Korevaar}
holds with $\theta(x) = \frac{x}{\log (x)}$ and $\omega(x) = \log (x)$ (up to asymptotic behavior, actually, since $\theta(x)$ is
more accurately defined in terms of the Lambert $W$-function).  Condition (11.2) of \cite{Korevaar}
follows from the application of the Mellin inversion formula to \eqref{eq:weighted_Eisen}, see
in particular Lemma 7.3 of \cite{JL93} where, in that notation, $\sigma = 0$ and $\sigma' = -\delta$ for some
$\delta > 0$. The conclusion of Theorem 11.1 of \cite{Korevaar}, as stated in (11.4),
then gives that
$$
\sum_{\genfrac{}{}{0pt}{}{\gamma\in\Gamma_\infty\setminus\Gamma}{\im(\gamma\tau) > \varepsilon}}
\left(\frac{1}{1-e^{2\pi i (z-\gamma\tau)}} - \frac{1}{1-e^{2\pi i (z-\overline{\gamma\tau})}}\right)
=\mathcal F_{1} (z,\tau) + O\left(\frac{1}{\log (\varepsilon)}\right)
\,\,\,\,\,
\text{\rm as $\varepsilon \rightarrow 0^{+}$}.
$$
This completes the proof of Theorem \ref{thm: main 2_2}.
\end{proof}

\section{Inner products}\label{sec:Inner_Products}

\subsection{A regularized inner product}\label{sec. regul inner prod}

Here, define a regularized inner product\footnote{As this extends Petersson's inner product for cusp forms, we choose the same symbol here.}. For $Y>0$, let
$$
P_Y:=\{z=z_1+iz_2: 0<z_1<1,\, z_2\geq Y\}.
$$
Assume that $Y>0$ is sufficiently large so that for $\ell\in\{1,\ldots,K\}$ the cuspidal zones $P_Y$ and $\sigma_\ell(P_Y)$ are disjoint.  Let
$$
\mathfrak F_{Y}:= \mathfrak F \Big\backslash \left( \bigcup\limits_{\ell=1}^K\sigma_\ell(P_Y)\cup P_Y \right).
$$
For $\mathfrak{z}_1, \ldots, \mathfrak{z}_\ell \in \mathfrak F$ and $\varepsilon_1, \ldots, \varepsilon_\ell\in\R^+$, define
$$
B_{\varepsilon_j}(\mathfrak{z}_j):= \left\{z\in \mathbb{H} : \left|\frac{z-\mathfrak{z}_j}{z-\overline{\mathfrak{z}_j}}\right|<\varepsilon_j\right\}.
$$
We then set
$$
\mathfrak F_{Y;\varepsilon_1, \ldots, \varepsilon_\ell } := \mathfrak F_{Y;\varepsilon_1, \ldots, \varepsilon_\ell }(\mathfrak{z}_1, \ldots, \mathfrak{z}_\ell) := \mathfrak F_{Y}\setminus \bigcup_{j=1}^{\ell} \left(B_{\varepsilon_j}(\mathfrak{z}_j)\cap \mathfrak F_Y\right).
$$
For functions $g$ and $h$ which satisfy weight $k\in\mathbb{Z}$ modularity and whose
singularities lie in $\{\mathfrak{z}_1, \ldots, \mathfrak{z}_\ell ,i\infty\}$, we define %(in case of the existence)
\begin{equation}\label{eq:regularized_integral}
\langle g,h\rangle:= \lim_{Y\to\infty}\lim_{\varepsilon_1\to0^+} \cdots \lim_{\varepsilon_\ell\to0^+} \int\limits_{\mathfrak F_{Y;\varepsilon_1, \ldots, \varepsilon_\ell }}g(w)h(w)w_2^k d\mu.
\end{equation}
We have the following lemma.

\begin{lemma}\label{lem. integrability of singularities}
  Let $h$, $H$ be meromorphic functions on $\H$ with (at most) isolated singularities at $\zf \in \H$. Assume that, with $c_{\mathfrak{z}}$, $C_\mathfrak{z}$ constants depending on $\zf$ and $h$ resp. $H$. Then we have
  $$
  h(z)=\frac{c_{\mathfrak{z}}}{(z-\zf)(z-\bar \zf)}+O(1), \quad H(z)= C_{\mathfrak{z}}\log|z-\zf| +O(1) \,\,\,\,\,\text{\rm as $z\to \zf$.}
  $$

  \begin{enumerate}[label=\rm(\arabic*),leftmargin=*]
  \item We have
  $$
  \lim_{\varepsilon \to 0^+}\int\limits_{B_\varepsilon(\zf)}h(z)d\mu=0.
  $$
  \item We have
 $$
  \lim_{\varepsilon \to 0^+}\int\limits_{B_\varepsilon(\zf)}h(z)H(z) d\mu=0.
  $$
  \end{enumerate}
\end{lemma}
\begin{proof}
  It suffices to prove (2). For $z\in B_\varepsilon(\zf)$ we write
  \begin{equation}\label{expand}
  h(z)= \frac{c_{\mathfrak{z}}}{(z-\zf)(z-\bar \zf)} + h_1(z)
  \,\,\,\,\,
  \text{\rm and}
  \,\,\,\,\, H(z)=C_{\mathfrak{z}}\log|z-\zf| + H_1(z),
  \end{equation}
  where $h_1(z), H_1(z) =O(1)$, as $z\to \zf$.
  We use polar coordinates centered at $\zf$, following \cite{Pe54}, pp. 41--43
  (with $r=0$ in the notation of \cite{Pe54})
\begin{equation}\label{eq. polar coord}
  \frac{z-\zf}{z-\bar \zf}=R e^{i\theta}, \quad d\mu(z)=\frac{4RdRd\theta}{\left(1-R^2\right)^{2}}.
  \end{equation}
  This yields that
  %{\bf KB: Wouldn't one also prove this with polar coordinates? If yes, we should. LS: done}
   $$
   \lim_{\varepsilon \to 0^+}\int\limits_{B_\varepsilon(\zf)}\log|z-\zf|d\mu(z)=\lim_{\varepsilon \to 0^+}\int\limits_{0}^\varepsilon \int\limits_0^{2\pi}\log\bigg|\frac{2\zf_2R}{1-Re^{i\theta}}\bigg|
  \frac{4RdRd\theta}{\left(1-R^2\right)^{2}}=0.
   $$
Using \eqref{expand}, we therefore have, as $z\to\zf$,
\begin{equation*}
  \lim_{\varepsilon \to 0^+}\int\limits_{B_\varepsilon(\zf)}h(z)H(z) d\mu=  c_{\mathfrak{z}} C_{\mathfrak{z}}\lim_{\varepsilon \to 0^+}\int\limits_{B_\varepsilon(\zf)} \frac{\log|z-\zf|}{(z-\zf)(z-\bar \zf)}d\mu(z)+ c_\mathfrak{z}\lim_{\varepsilon \to 0^+}\int\limits_{B_\varepsilon(\zf)} \frac{H_1(z)d\mu(z)}{(z-\zf)(z-\bar \zf)}.
  \end{equation*}
   Using again \eqref{eq. polar coord} in the first integral,
  we get
  $$
  \int\limits_{B_\varepsilon(\zf)} \frac{\log|z-\zf|}{(z-\zf)(z-\bar \zf)}d\mu(z)=-\frac{1}{4\zf_2^2}\int\limits_{0}^\varepsilon \int\limits_0^{2\pi} \frac{\left(1-Re^{i\theta}\right)^2}{Re^{i\theta}}\log\left|\frac{2\zf_2R}{1-Re^{i\theta}}\right|
  \frac{4RdRd\theta}{\left(1-R^2\right)^{2}}.
  $$
  As above, the integral vanishes as $\varepsilon \to 0^+$. Analogously, we deduce that
  $$
  \lim_{\varepsilon \to 0^+}\int\limits_{B_\varepsilon(\zf)} \frac{H_1(z)}{(z-\zf)(z-\bar \zf)}d\mu(z)=0.
  $$
  Thus, (2) is established, and the lemma follows.
\end{proof}

\begin{remark}\label{rem. integrability}
  \Cref{lem. integrability of singularities} shows that meromorphic functions on $\H$ having singularities only at finitely many isolated points $\zf_1,\ldots,\zf_\ell \in \H$ such that their asymptotic behavior as $z\to \zf_j$ (for $j\in\{1,\ldots,\ell\}$), is  $\frac{c_{\mathfrak{z}_j}}{(z-\zf_j)(z-\overline{\zf_j})}$ or  $\frac{C_{\zf_j}}{(z-\zf_j)(z-\overline{\zf_j} )}\log|z-\zf_j|$ can be
  considered as being integrable in the sense of \eqref{eq:regularized_integral}.
\end{remark}

\subsection{Expressing $\mathbb{F}(z,\tau)$ as inner product}

%\textcolor{blue}{This  should be extended to cover all Fuchsian groups - proof analysis to be completed}
We express $\mathbb{F}(z,\tau)$ as a regularized inner product.

\begin{proposition}\label{prop:F_inner_product_firstformula}
The series $\mathbb{F}(z,\tau)$ converges absolutely for all $z$ such that $z_2>\im(\gamma \tau)$ for all $\gamma\in\Gamma$.
Moreover, for such $z$, we have
\begin{equation} \label{eq. gen series as inner product}
\mathbb{F}(z,\tau)=-\left\langle R_{0,z}(\mathcal{G}(z,\cdot)) - \frac1{\operatorname{vol}(\Gamma\backslash \mathbb{H})z_2},  \mathcal{G}(\cdot,\tau)\right\rangle.
\end{equation}
\end{proposition}
\begin{proof}
  The absolute convergence of $\mathbb F(z,\tau)$ follows from \eqref{eq. NP deriv bound0}.
  To show \eqref{eq. gen series as inner product} we follow the proof of Theorem 1 of \cite{CJS23}.
First, we prove that, for $n\in\N$ and $\tau\in \mathfrak F$,
  $$
\langle F_{-n,1}, \mathcal{G}(\cdot,\tau) \rangle=-\left[\frac{\partial}{\partial s}F_{-n,s}(\tau)\right]_{s=1}.
$$
 We assume that $K\geq 1$.
 Applying Green's formula if $K\geq 1$ yields additional terms along the horocycles $\sigma_\ell P(Y)$ ($\ell\in\{1,\ldots,K\}$).
 Hence, (39) from \cite{CJS23} becomes
\begin{align}
  &\langle F_{-n,1}(\cdot), \mathcal{G}(\cdot,\tau) \rangle= \lim_{Y\to \infty} \int\limits_{\mathfrak F_Y} F_{-n,1}(w)\mathcal{G}(w,\tau)d\mu(w)\label{eq. inner prod general}\\
  &=\lim_{Y\to\infty} \left[\frac{\partial}{\partial s}\left(\int\limits_0^1 \left(\left[\frac{\partial}{\partial w_2}F_{-n,1}(w)\right]_{w_2=Y}\mathcal{G}_s(w,\tau) - F_{-n,1}(w)\left[\frac{\partial}{\partial w_2}\mathcal{G}_s(w,\tau)\right]_{w_2=Y} \right) dw_1\right)\right]_{s=1}\nonumber\\
  +\!&\sum_{\ell=1}^K \lim_{Y\to\infty}\! \left[\frac{\partial}{\partial s}\!\left(\int_0^1\!\! \left(\left[\frac{\partial}{\partial\mathrm n}F_{-n,1}(\sigma_\ell w)\right]_{\!w_2=Y}\hspace{-0.25cm}\mathcal{G}_s(\sigma_\ell w,\tau) \!-\! F_{-n,1}(\sigma_\ell w)\!\left[\frac{\partial}{\partial\mathrm n}\mathcal{G}_s(\sigma_\ell w,\tau)\right]_{\!w_2=Y} \right) \!dw_1\right)\right]_{\!s=1},\nonumber
\end{align}
where $\mathcal{G}_s(w,\tau):=G_s(w,\tau)+E(\tau;s)$.
The sum over $\ell$ is estimated by combining Corollary \ref{cor. green's bounds other cusps}  with \eqref{eq. bound for NP at other cusps} and \eqref{eq. bound for norm der NP at other cusps} to deduce that
\begin{multline*}
\frac{\partial}{\partial s}\!\left(\int_0^1\!\! \left(\left[\frac{\partial}{\partial\mathrm n}F_{-n,1}(\sigma_\ell w)\right]_{\!w_2=Y}\hspace{-0.25cm}\mathcal{G}_s(\sigma_\ell w,\tau) \!-\! F_{-n,1}(\sigma_\ell w)\!\left[\frac{\partial}{\partial\mathrm n}\mathcal{G}_s(\sigma_\ell w,\tau)\right]_{\!w_2=Y} \right) \!dw_1\right)_{\!s=1}\\ =O\left(\frac{P_{1}(\tau)+ \left(P_{\varrho_\ell,i\infty}(\tau)+\!P_{\ell,0,1}(\tau)\right) \log (Y)+\log^2(Y)  }{Y}\right).
\end{multline*}

Now we investigate the first term, by analyzing\footnote{The Fourier expansions at  $i\infty$ for $K=0$ and $K\geq 1$ are the same, so the arguments using those expansions from \cite{CJS23} also apply for $K\geq 1$.} the proof of Theorem 1 of \cite{CJS23}. Namely, there it was proved that, for $Y>0$ sufficiently large, the term
  $$
\left[\frac{\partial}{\partial s}\left(\int\limits_0^1 \left(\left[\frac{\partial}{\partial w_2}F_{-n,1}(w)\right]_{w_2=Y}\mathcal{G}_s(w,\tau) - F_{-n,1}(w)\left[\frac{\partial}{\partial w_2}\mathcal{G}_s(w,\tau)\right]_{w_2=Y} \right) dw_1\right)\right]_{s=1}
  $$
equals a sum of three terms (see (40) of \cite{CJS23})  and each of those three terms, which are in our notation $[\frac{\partial}{\partial s}T_k(Y,s;\tau)]_{s=1}$ ($k\in\{1,2,3\}$), were evaluated. From the bounds on those three terms, one can see that for $Y> 2\tau_2+\frac4{c_{\Gamma}}$ (with $c_\Gamma$ the smallest positive left lower entry of matrices from $\Gamma$) one has
$$
	\left[\frac{\partial}{\partial s}T_1(Y,s;\tau)\right]_{s=1}	=-\left[\frac{\partial}{\partial s}F_{-n,s}(\tau)\right]_{s=1}\left(1+O\left(\frac{1}{nY}\right)\right),
$$
where the implied constant is independent of $\tau$. Next we have
$$
\left[\frac{\partial}{\partial s}T_2(Y,s;\tau)\right]_{s=1}	=O\left(\frac{P(\tau)}{Y}\right),\quad
\left[\frac{\partial}{\partial s}T_3(Y,s;\tau)\right]_{s=1}=O\left(\frac{e^{n\tau_2}}{Y}\right).
$$
Therefore, by \eqref{eq. NP deriv bound} and \eqref{eq. inner prod general}, using that $P$, $P_1$, and $P_{\varrho_\ell,i\infty}$ grow sub-exponentially in $\tau_2$,
$$
\int\limits_{\mathfrak F_Y} F_{-n,1}(w)\mathcal{G}(w,\tau)d\mu(w) = - \left[\frac{\partial}{\partial s}F_{-n,s}(\tau)\right]_{s=1} + O\left(e^{n\tau_2}\frac{\log^2 (Y)}{Y}\right),
$$
for $Y> 2\tau_2+\frac4{c_{\Gamma}}$, as $Y\to \infty$, where the implied constant does not depend on $z$ and $\tau$. Hence, for $z$ satisfying $z_2>\im(\gamma \tau)$ for all $\gamma\in\Gamma$, we have, as $Y\to \infty$,
$$
2\pi\sum_{n\ge1} \sqrt{n}\int\limits_{\mathfrak F_Y} F_{-n,1}(w)\mathcal{G}(w,\tau)d\mu(w)e^{2\pi i nz}= -\mathbb{F}(z,\tau)+O\left(\frac{\log^2 (Y)}{Y}\right).
$$
If we fix $Y> 2\tau_2+\frac{4}{c_{\Gamma}}$ sufficiently large, then it is immediate from \cite{Ne73} that
$|\sqrt{n}F_{-n,1}(w)e^{2\pi i nz}|$ decays exponentially in $n$, for $z_2>Y$. Thus, by the Dominated Convergence Theorem,
we conclude
$$
2\pi\sum_{n\ge1}\sqrt{n}\int\limits_{\mathfrak F_Y} F_{-n,1}(w)\mathcal{G}(w,\tau)d\mu(w)e^{2\pi i nz}= \int\limits_{\mathfrak F_Y}\mathcal F(z,w)\mathcal{G}(w,\tau)d\mu(w).
$$
Therefore, we have
$$
\int\limits_{\mathfrak F_Y}\mathcal F(z,w)\mathcal{G}(w,\tau)d\mu(w)=-\mathbb{F}(z,\tau)+
O\left(\frac{\log^2 (Y)}{Y}\right).
$$
Hence, we obtain
$$
\mathbb{F}(z,\tau)=-\langle \mathcal F(z,\cdot), \mathcal{G}(\cdot ,\tau) \rangle.
$$
This, combined with Theorem \ref{thm: main 1} (1), the proof follow from analytic continuation.
\end{proof}

\section{Proof of Theorem \ref{thm. gen series of deriv}}\label{sec:Proof_of_gen_series_deriv}
In this section we prove \Cref{thm. gen series of deriv}.

\begin{proof}[Proof of \Cref{thm. gen series of deriv}]

As described in \eqref{eq. greens function log singul expr}, $\mathcal{G}(w,\tau)$ has a logarithmic singularity in the neighbourhood of
$\tau \in\mathfrak F$. Combined with arguments from Subsection 5.3 of \cite{CJS23} we get that
	\begin{equation}\label{eq. inner of 1 and G}
	\langle 1, \mathcal{G}(\cdot,\tau) \rangle= \vol(\Gamma\backslash \mathbb{H})\beta - P(\tau)
	\end{equation}
which corresponds to the last displayed equation in Subsection 5.3 of \cite{CJS23}.
We now discuss this expression in further detail.
In the notation of \cite{Ar15}, the resolvent kernel $G_s(z,w)$ admits a Laurent expansion of the form\footnote{The constant term $g(z,w)$ is called the {\it hyperbolic Green's function}. It is an $L^2$-function which is orthogonal to constant functions.}
	$$
	G_s(z,w)=\frac{1}{\operatorname{vol}(\Gamma\backslash \mathbb{H})s(1-s)}-\frac{g(z,w)}{4\pi}+O(s-1)
\,\,\,\,\,\text{\rm as $s\to1$.}
	$$
Combining this with \eqref{KronLimitPArGen} we deduce that
	\begin{equation}\label{eq. G in terms of g}
	\mathcal{G}(z,w)= -\frac{g(z,w)}{4\pi}+\beta-\frac{P(w)}{\vol(\Gamma\backslash \mathbb{H})}.
	\end{equation}
Therefore, since $g(z,w)$ is an $L^2$-function which is orthogonal to constant functions, we have
$$
\langle 1, \mathcal{G}(\cdot,\tau) \rangle = \int\limits_{\mathfrak{F}}\left(-\frac{g(w,\tau)}{4\pi}+\beta-\frac{P(\tau)}{\vol(\Gamma\backslash \mathbb{H})}\right)d\mu(w)= \vol(\Gamma\backslash \mathbb{H})\beta - P(\tau),
$$
which proves \eqref{eq. inner of 1 and G}.
Continuing, we can combine \eqref{eq. inner of 1 and G} with Proposition \ref{prop:F_inner_product_firstformula} to get that
	\begin{equation}\label{eq. mathbb F first expr}
	\mathbb{F}(z,\tau)= -\langle R_{0,z}(\mathcal{G}(z,\cdot)), \mathcal{G}(\cdot,\tau)\rangle +\frac{\beta}{z_2}-\frac{P(\tau)}{\vol(\Gamma\backslash \mathbb{H})z_2}.
	\end{equation}

We next study \eqref{eq. mathbb F first expr} in detail.	From \eqref{eq. G in terms of g} we get
	$$
	R_{0,z}(\mathcal{G}(z,w))= R_{0,z}\left(-\frac{g(z,w)}{4\pi}\right),
	$$
	Hence, we can write
	\begin{align} \label{eq. change inner prod and deriv}\notag
	-\langle R_{0,z}(\mathcal{G}(z,\cdot)), \mathcal{G}(\cdot,\tau)\rangle &= \left\langle \frac{1}{4\pi}R_{0,z}(g(z,\cdot)), -\frac{g(\cdot,\tau)}{4\pi}+ \beta-\frac{P(\tau)}{\vol(\Gamma\backslash \mathbb{H})} \right\rangle
	\\&\!\!\!\!\!\! = -\frac{1}{16\pi^2}\langle R_{0,z}(g(z,\cdot)), g(\cdot,\tau)\rangle+\frac{1}{4\pi} \left( \beta-\frac{P(\tau)}{\vol(\Gamma\backslash \mathbb{H})}\right)\langle R_{0,z}(g(z,\cdot)),1 \rangle.
	\end{align}
Now, we prove that the term in  \eqref{eq. change inner prod and deriv} vanishes.
By Theorem \ref{thm: main 1} (2), the function
$$
w\mapsto R_{0,z} (\mathcal{G}(z, w))= R_{0,z}\left(-\frac{g(z,w)}{4\pi}\right)
$$
is a weight zero polar harmonic Maass form.  Furthermore, by
Theorem \ref{thm: main 1} (3), its only pole in $\mathfrak F$ is a simple pole at $w=z$.  As such, we have, as $w\to z$,
	$$
	R_{0,z}  \left(\mathcal{G}(z, w)\right)= -\frac{\mathrm{Stab}_z z_2}{\pi (w-z)(w-\bar z)} + O(1).
	$$
	By Lemma \ref{lem. integrability of singularities} this means that $R_{0,z} (\mathcal{G}(z, w))$ is bounded by an integrable function on $\mathfrak F$, uniformly so in $z$. Therefore, the interchange of derivative in $z$ and the integral is justified, which gives that
	\begin{equation*}%\label{eq. inner prod Rg with 1}
	\langle R_{0,z}(g(z,\cdot)),1 \rangle= \int\limits_{\mathfrak F} R_{0,z}(g(z,w))d\mu(w)=R_{0,z}\left(\,\int\limits_{\mathfrak F} g(z,w)d\mu(w)\right)=0.
	\end{equation*}
To be precise, we use that,  by Remark \ref{rem. integrability}, the regularized integral
\begin{equation}\label{eq:reg_int}
\int\limits_{\mathfrak F} g(z,w)d\mu(w)
\end{equation}
equals the integral over $\mathfrak F$ and \eqref{eq:reg_int} vanishes,
see \cite{Ar15}. Continuing, the above shows that
	$$
	-\langle R_{0,z}(\mathcal{G}(z,\cdot)), \mathcal{G}(\cdot,\tau)\rangle=-\frac{1}{16\pi^2}\langle R_{0,z}(g(z,\cdot)), g(\cdot,\tau)\rangle.
	$$
	We next claim that one can interchange $R_{0,z}$ with the regularized integral, so then
	\begin{equation}\label{eq. change inner prod and deriv 1}
	-\langle R_{0,z}(\mathcal{G}(z,\cdot)), \mathcal{G}(\cdot,\tau)\rangle=-\frac{1}{16\pi^2}R_{0,z}\left(\langle g(z,\cdot), g(\cdot,\tau)\rangle\right).
	\end{equation}
	To prove \eqref{eq. change inner prod and deriv 1} it suffices to show that $H(w) := |R_{0,z}(g(z,w)g(w,\tau))|$
	is bounded by an integrable function, uniformly in $z$ and $\tau$.
	Indeed, for $z\neq \tau$, $H$ is bounded outside a union of two small hyperbolic discs centered at
$z$ and $\tau$. Furthermore, as $w\to z$, $H$ is dominated by a constant multiple of $|(w-z)(w-\bar z)|^{-1}$,
which is integrable by Lemma
\ref{lem. integrability of singularities}. Also, as $w\to \tau$, $H$ is bounded by a constant multiple of $\log|w-\tau|$ which is
also integrable. If $z=\tau$, then $H$ is bounded outside a small hyperbolic disc centered at $z=\tau$,
while as $w\to z$, if $z=\tau$, $H$ is dominated by a constant multiple of $|(w-z)(w-\bar z)|^{-1}|\log|w-z||$, which is integrable.
Thus $H$ is bounded by an
integrable function,  Therefore, the interchange is allowed.
	
	To complete the proof of Theorem \ref{thm. gen series of deriv},
we compute the inner product on the right-hand side of \eqref{eq. change inner prod and deriv 1}. To do so, we use the expression of $\frac{1}{4\pi}g(\cdot,\tau)$ in terms of the heat kernel on $M=\Gamma\backslash \mathbb{H}$. Namely, according to \cite{Ar15} we have
for $z\neq w$ that
	$$
	\frac{g(z,w)}{4\pi}=\int\limits_0^\infty\left(K_M(t;z,w) - \frac{1}{\vol(\Gamma\backslash \mathbb{H})}\right)dt
	$$
	and the integral converges absolutely.\footnote{This follows from the spectral expansion of the heat kernel ((19) of \cite{Ar15}) and the fact that the constant term in the spectral expansion equals $\frac{1}{\vol(\Gamma\backslash \mathbb{H})}$.} Since the heat kernel is real-valued we deduce
	\begin{align*}
	\frac{1}{16\pi^2}\langle g(z,\cdot),g(\cdot,\tau) \rangle &=\int\limits_M \int\limits_0^\infty\left(K_M(t;z,w) - \frac{1}{\vol(\Gamma\backslash \mathbb{H})}\right)dt \int\limits_0^\infty\left(K_M(x;w,\tau) - \frac{1}{\vol(\Gamma\backslash \mathbb{H})}\right)dxd\mu(w)\\
	&=\int\limits_0^\infty\int\limits_0^\infty \int\limits_M  \left(K_M(t;z,w) - \frac{1}{\vol(\Gamma\backslash \mathbb{H})}\right) \left(K_M(x;w,\tau) - \frac{1}{\vol(\Gamma\backslash \mathbb{H})}\right)d\mu(w) dx dt,
	\end{align*}
	where the absolute convergence of the integrals in the variables $t,\, x$, for $w\in M\setminus\{ z,\tau\}$  justifies the application of the Fubini--Tonelli Theorem. The semigroup property for the heat kernel states that
	$$
	\int\limits_M K_M(t;z,w)K_M(x;w,\tau) d\mu(w)=K_M(t+x;z,\tau).
	$$
	Also, we have
	$$
	\int\limits_M K_M(t;z,w) d\mu(w)= \int\limits_M K_M(x;w,\tau) d\mu(w)=1.
	$$
	Therefore\footnote{For more details on the semigroup property and stochastic completeness of the heat kernel on manifolds, see e.g. \cite{Gr99}, (2.11) and Subsection 3.3.}, we have
	\begin{equation*}
	\frac{1}{16\pi^2}\langle g(z,\cdot),g(\cdot,\tau) \rangle =\int\limits_0^\infty\int\limits_0^\infty \left(K_M(t+x;z,\tau) - \frac{1}{\vol(\Gamma\backslash \mathbb{H})}\right)dtdx.
	\end{equation*}
	
	Now we use the spectral expansion of the heat kernel; see for example \cite{Ar15} and references therein.
Specifically, for $z\neq \tau$ we have that
	\begin{multline*}
	K_M(t;z,\tau)= \sum_{j\ge0}  u_j(z)u_j(\tau)e^{-\lambda_jt}+\frac{1}{4\pi}\int\limits_{\R} E\left(z;\frac12+ir\right)E\left(\tau;\frac12-ir\right)e^{-\left(\frac{1}{4}+r^2\right)t}dr\\ + \sum_{\ell=1}^K \frac{1}{4\pi}\int\limits_{\R} E_{\varrho_\ell}\left(z;\frac12+ir\right)E_{\varrho_\ell}\left(\tau;\frac12-ir\right)e^{-\left(\frac{1}{4}+r^2\right)t}dr,
	\end{multline*}
	where $0=\lambda_0<\lambda_1<...$ is the sequence of the discrete eigenvalues of the Laplacian $\Delta_0$ on $M$ and $u_j(z)$
are the associated real-valued eigenfunctions, normalized so that their $L^2$-norm equals one.
Note that if there are no cusps other than $i\infty$, then we consider the sum over $\ell$ to be identically zero.
In particular, we have
	\begin{align*}
	K_M(t+x;z,\tau) &- \frac{1}{\vol(\Gamma\backslash \mathbb{H})}\\
	= \sum_{j\ge1}& e^{-\lambda_j(t+x)} u_j(z)u_j(\tau)+\frac{1}{4\pi}\int\limits_{\R} E\left(z;\frac12+ir\right)E\left(\tau;\frac12-ir\right)e^{-\left(\frac{1}{4}+r^2\right)(t+x)}dr\\
	&\hspace{3.3cm}+ \sum_{\ell=1}^K \frac{1}{4\pi}\int\limits_{\R} E_{\varrho_\ell}\left(z;\frac12+ir\right)E_{\varrho_\ell}\left(\tau;\frac12-ir\right)
	e^{-\left(\frac{1}{4}+r^2\right)(t+x)}dr,
	\end{align*}
so then
	\begin{align}\label{eq. spectral exp}\notag
	\int\limits_0^\infty\int\limits_0^\infty &\left(K_M(t+x;z,\tau) - \frac{1}{\vol(\Gamma\backslash \mathbb{H})}\right)dtdx \\&
\notag=\sum_{j\ge1} h(r_j) u_j(z)u_j(w)+\frac{1}{4\pi}\int\limits_{\R} E\left(z;\frac12+ir\right)E\left(\tau;\frac12-ir\right)h(r)dr\\&\hspace{5.6cm}+ \sum_{\ell=1}^K \frac{1}{4\pi}\int\limits_{\R} E_{\varrho_\ell}\left(z;\frac12+ir\right)E_{\varrho_\ell}\left(\tau;\frac12-ir\right)h(r)dr,
	\end{align}
	where
%{\color{red}\bf KB: The $\lambda_j$ are already used above. LS: Yes, but here we define $r_j$ used in the above display in terms of the discrete eigenvalues $\lambda_j$ through KB: Ah. The definition was on the wrong side - changed.}
$\lambda_j=:\tfrac1{4}+r_j^2$ and $h(t):=(\tfrac1{4}+t^2)^{-2}$.
	The function $h$ satisfies condition (1.63) of \cite{Iwa02}. Hence, according to Theorem 7.4 of \cite{Iwa02}, the right-hand side of
\eqref{eq. spectral exp} is the spectral expansion of an automorphic kernel
$$
\mathcal{K^*}(z,\tau)=\sum_{\gamma\in \Gamma} k^* (u(z,\gamma \tau))
$$
	associated to the point pair invariant ${k^*}(u)$ which is the inverse Selberg/Harish--Chandra transform of $h$. In other words, the above
arguments prove that
	$$
	-\langle R_{0,z}(\mathcal{G}(z,\cdot)), \mathcal{G}(\cdot,\tau)\rangle= -R_{0,z}\left(\mathcal{K^*}(z,\tau)\right).
	$$

	From \eqref{eq. mathbb F first expr}, to order to complete the proof of Theorem \ref{thm. gen series of deriv},
we need to show that ${k^*}(u)=k(u)$, where $k(u)$ is defined in \Cref{thm. gen series of deriv}. By (1.64')  of \cite{Iwa02}
	$$
	{k^*}(u)= \frac{1}{4\pi}\int\limits_{\R} F_{\frac12+it}(u)\frac{t\tanh(\pi t)}{\left(\tfrac1{4}+t^2\right)^2}dt,
	$$
	where $F_s(u):={}_2F_1(s,1-s,1;u)$. By \cite{Iwa02}, in particular by the paragraph above (1.43) we have that
$F_s(u)=P_{-s}(1+2u)$  where $P_\nu$ denotes the Legendre function.  Combining with 8.733 of \cite{GR07}, we have that $F_{\frac12+it}(u)=P_{-\frac12+it}(1+2u)$. Hence, for $z\neq \tau$ we obtain
	\begin{equation}\label{eq. widetilde k}
	{k^*}(u(z,\tau)) \!=\! \frac{1}{2\pi} \!\int\limits_{0}^{\infty}\!\! P_{\!-\frac12+it}(1+2u(z,\tau))\frac{t\tanh(\pi t)}{\left(\frac{1}{4}+t^2\right)^2}dt
	 \!=\! \frac{1}{2\pi} \!\int\limits_{0}^{\infty}\!\! P_{\!-\frac12+it}(\cosh(d(z,\tau)))\frac{t\tanh(\pi t)}{\left(\frac{1}{4}+t^2\right)^2}dt,
	\end{equation}
	where we use that $1+2u(z,\tau)= \cosh(d(z,\tau))$, see (1.3) from \cite{Iwa02}.
	It is left to evaluate the integral on the right-hand side. By 7.213 of \cite{GR07}, we have
$$
	\int\limits_{0}^{\infty} P_{-\frac12+it}(\cosh(b))\frac{t\tanh(\pi t)}{a^2+t^2}dt=Q_{a-\frac12}(\cosh(b)),
$$
	which is valid for Re$(a)>0$. In particular, for $a\in\R^+$, we obtain that
\begin{align*}
	\frac{1}{2a}\frac{\partial}{\partial a}\int\limits_{0}^{\infty} P_{-\frac12+it}(\cosh(b))\frac{t\tanh(\pi t)}{a^2+t^2}dt = -\int\limits_{0}^{\infty} P_{-\frac12+it}(\cosh(b))\frac{t\tanh(\pi t)}{\left(a^2+t^2\right)^2}dt =\frac{1}{2a}\frac{\partial}{\partial a}Q_{a-\frac12}(\cosh(b)).
\end{align*}
Therefore, we have
	$$
	\int\limits_{0}^{\infty} P_{-\frac12+it}(\cosh(d_{\hyp}(z,\tau)))\frac{t\tanh(\pi t)}{\left(\frac{1}{4}+t^2\right)^2}dt=-\frac{1}{2a}\left[\frac{\partial}{\partial a}Q_{a-\frac12}(\cosh (d_{\hyp}(z,\tau)))\right]_{a=\frac12}.
	$$
	From \eqref{eq. widetilde k} we then get that
	$$
	{k^*}(u(z,\tau))=-\frac{1}{2\pi}\left[\frac{\partial}{\partial \nu}Q_\nu(1+2u(z,\tau))\right]_{\nu=0}.
	$$
Equation (3.8a) of \cite{Sm17} states that, for $w\in\mathbb{C}\setminus[-1,1]$,
	\begin{align*}
	\left[\frac{\partial}{\partial \nu}Q_\nu(w)\right]_{\nu=0}= -\mathrm{Li}_2\left(\frac{1-w}{2}\right) -\frac12\log\left(\frac{w+1}{2}\right)\log\left(\frac{w-1}{2}\right) -\frac{\pi^2}{6}.
	\end{align*}
If take $w=1+2u(z,\tau)>1$, we arrive at the expression
	$$
	{k^*}(u(z,\tau))=\frac{1}{2\pi}\left(\mathrm{Li}_2(-u)+\frac12\log(1+u)\log(u) +\frac{\pi^2}{6}\right).
	$$
	By Lemma \ref{lem. dilog}, ${k^*}(u(z,\tau))$ equals the expression in \Cref{thm: main 1}, completing the proof.\qedhere
\end{proof}

\end{document}